\documentclass[12pt]{article}
\usepackage{times}
\usepackage{amsmath,amsfonts,amstext,amssymb,amsbsy,amsopn,amsthm,eucal,slashed,amstext,enumerate,accents}

\usepackage{color}
\usepackage{txfonts}
\usepackage{dsfont}
\usepackage{hyperref}
\usepackage{graphicx}   

\numberwithin{equation}{section}
  \theoremstyle{plain}
 \newtheorem{theorem}[equation]{Theorem}
 
\newtheorem{proposition}[equation]{Proposition}
 
 \newtheorem{corollary}[equation]{Corollary}

 \theoremstyle{remark}

 \newtheorem{remark}[equation]{Remark}

\theoremstyle{definition}

\newtheorem{example}[equation]{Example}

\newcommand{\abs}[1]{\lvert#1\rvert}

\newcommand{\Rc}{{\rm Ric}}
\newcommand{\Ric}{{\rm Ric}}

\newcommand{\dR}{\mathds{R}}

\newcommand{\p}{\parallel}
\newcommand{\Rm}{\text{Rm}}

\newcommand{\eps}{\varepsilon}

\renewcommand{\H}{\mathcal{H}}

\newcommand{\tr}{\text{tr}}
\newcommand{\Lap}{\Delta}

\begin{document}

\title{Ricci Curvature and \\ Bochner Formulas for Martingales}
\author{Robert Haslhofer and Aaron Naber\thanks{R.H. has been partially supported by NSF grant DMS-1406394 and NSERC grant RGPIN-2016-04331. A.N. has been partially supported by NSF grant DMS-1406259. Both authors also acknowledge the invitation to MSRI Berkeley in spring 2016 supported by NSF Grant DMS-1440140, where part of this research has been carried out.}}
\date{\today}

\maketitle
\begin{abstract}

We generalize the classical Bochner formula for the heat flow on $M$ to martingales on the path space $PM$, and develop a formalism to compute evolution equations for martingales on path space.  We see that our Bochner formula on $PM$ is related to two sided bounds on Ricci curvature in much the same manner that the classical Bochner formula  on $M$ is related to lower bounds on Ricci curvature. Using this formalism, we obtain new characterizations of bounded Ricci curvature, new gradient estimates for martingales on path space, new Hessian estimates for martingales on path space, and streamlined proofs of the previous characterizations of bounded Ricci curvature from \cite{Naber_char}.
\end{abstract}

\small{\tableofcontents}


\section{Introduction}

The main goal of this paper is to explain how bounded Ricci curvature can be understood by analyzing the evolution of martingales on path space, generalizing the well known and important principles of how lower bounds on Ricci curvature can be understood by analyzing the heat flow.  The formalism we develop will allow us to do analysis on the path space $PM$ of manifolds with bounded Ricci curvature using techniques and ideas which mimic closely the ideas used to do analysis on manifolds with lower Ricci curvature bounds.\\

\subsection{Background on Lower and Bounded Ricci Curvature}


\noindent\textbf{Lower bounds for Ricci curvature.}
To put things into context, let us briefly mention the theory of  spaces with Ricci curvature bounded below, which has been a very active area of research in the last 30 years. This theory can be pursued either in the setting of smooth Riemannian manifolds and their Gromov-Hausdorff limits, see e.g. \cite{CheegerColding_almost,CheegerColding_structure,CheegerNaber_quant}, or in the more general setting of metric measure spaces, see e.g. \cite{LottVillani,Sturm,AGS,Gigli_splitting}.
The starting point for most of the analysis of such spaces with Ricci curvature bounded below, say by a constant $-\kappa$, is the classical Bochner inequality.  For solutions $H_tf$ of the heat flow this may be written as
\begin{align}\label{bochner_inequ}
\big(\partial_t - \tfrac{1}{2}\Lap\big) \abs{\nabla H_t f}^2 \leq - \abs{\nabla^2 H_t f}^2 + \kappa\abs{\nabla H_t f}^2\, .
\end{align}
The Bochner inequality \eqref{bochner_inequ} in particular implies the dimensional Bochner inequality
\begin{align}
\big(\partial_t - \tfrac{1}{2}\Lap\big) \abs{\nabla H_t f}^2  \leq - \tfrac{1}{n}\abs{\Delta H_t f}^2 + \kappa\abs{\nabla H_t f}^2\, ,
\end{align}
and the weak Bochner inequality
\begin{align}\label{weak_bochner}
\big(\partial_t - \tfrac{1}{2}\Lap\big) \abs{\nabla H_t f}^2  \leq  \kappa\abs{\nabla H_t f}^2\, ,
\end{align}
and conversely there is a self-improvement mechanism that allows one to go from \eqref{weak_bochner} to \eqref{bochner_inequ}, see \cite{Savare_selfimprovement,Sturm_selfimprovement}. Moreover, it is an interesting feature that all the above inequalities are in fact equivalent to the lower Ricci curvature bound. Using the Bochner inequality it is then a simple exercise to show that Ricci bounded below by $-\kappa$ is also equivalent to several other geometric-analytic estimates, e.g. that $e^{-\frac{\kappa}{2}t}|\nabla H_t f|$ is a subsolution to the heat flow, the sharp gradient 
\begin{equation}\label{heat_grad_est}
\abs{\nabla H_t f}\leq e^{\frac{\kappa}{2}t}H_t \abs{\nabla f}
\end{equation}
for the heat flow, as well as a sharp log-Sobolev inequality, a sharp spectral gap, etc; see e.g. \cite{BakryEmery,BakryLedoux} for much more on that.\\

\noindent\textbf{Characterizations of bounded Ricci curvature.}
In contrast to the well developed theory of Ricci curvature bounded below, until recently there was no characterization available at all for spaces with bounded Ricci curvature. This characterization problem has been solved recently by the second author \cite{Naber_char}. The key insight was that to understand two-sided bounds for Ricci curvature, and not just lower bounds, one should do analysis on path space $PM$, instead of analysis on $M$. By definition, given a complete Riemannian manifold $M$, its path space $PM=C([0,\infty),M)$ is the space of continuous curves in $M$. Path space comes equipped with a family of natural probability measures, the Wiener measure $\Gamma_x$ of Brownian motion starting at $x\in M$. Path space also comes equipped with a natural one parameter family of gradients, the $s$-parallel gradients $\nabla_s^\parallel$ ($s\geq 0$), which are given by considering derivatives of a function $F$ by vector fields which are parallel past time $s$, see Section \ref{ss:prelim:gradients} for precise definitions. Using this framework, it was proved in \cite{Naber_char} that the Ricci curvature of $M$ is bounded by a constant $\kappa$ if and only if the sharp gradient estimate
\begin{equation}\label{inf_grad_est_intro}
\left|\nabla_x \int_{PM} F\, d\Gamma_x\right| \leq \int_{PM}\left(\abs{\nabla_0^\parallel F}+\int_0^\infty\frac{\kappa}{2}e^{\kappa s/2}\abs{\nabla_s^\parallel F}\, ds\right) \, d\Gamma_x
\end{equation}
holds for all test functions $F:PM\to\mathbb{R}$. In the simplest case of one-point test functions, i.e. functions of the form $F(\gamma)=f(\gamma(t))$ where $f:M\to\mathbb{R}$ and $t$ is fixed, the infinite dimensional gradient estimate \eqref{inf_grad_est_intro} reduces to the finite dimensional gradient estimate \eqref{heat_grad_est}. Of course, one can also consider test functions depending on more than one single time, and this is one of the reasons why the infinite dimensional gradient estimate \eqref{inf_grad_est_intro} is strong enough to characterize two-sided Ricci bounds, and not just lower bounds. Further characterizations of bounded Ricci curvature have been obtained in terms of a sharp log-Sobolev inequality on path space and a sharp spectral gap for the Ornstein-Uhlenbeck operator, see \cite{Naber_char}. These ideas have  been implemented also in the parabolic setting to characterize solutions of the Ricci flow \cite{HaslhoferNaber_flow}. Another interesting variant of the characterizations of bounded Ricci curvature from \cite{Naber_char} has been obtained recently by Fang-Wu \cite{FangWu} and Wang-Wu \cite{WangWu}.\\

\subsection{Bochner Formula for Martingales}

\noindent\textbf{Generalizing the Bochner formula.} While \cite{Naber_char} gives a way to generalize certain estimates for lower Ricci curvature on $M$ to estimates for bounded Ricci curvature on $PM$, e.g. the finite dimensional gradient estimate \eqref{heat_grad_est} to the infinite dimensional gradient estimate \eqref{inf_grad_est_intro}, what hasn't been answered yet is the following question:\\

\emph{Is there any way to generalize the fundamental Bochner inequality \eqref{bochner_inequ} from $M$ to $PM$?}\\

This question has been the guiding principle for the present paper. Given that the Bochner inequality is the starting point for most of the theory of lower Ricci curvature, such a generalization would be clearly very valuable for the theory of bounded Ricci curvature. 
As we will see, the question does not amount to a straightforward translation (e.g. a first naive guess would be to simply replace the Laplacian on $M$ by the Ornstein-Uhlenbeck operator on $PM$), but in fact led us to reconsider some of the most basic aspects of stochastic analysis, such as the martingale representation theorem and submartingale inequalities.\\

\noindent\textbf{Martingales on path space.} The first main point we wish to explain is that martingales on $PM$ are the correct generalization of the heat flow on $M$. To describe this, given a complete Riemannian manifold $M$, consider its path space $PM=C([0,\infty),M)$ equipped with the Wiener measure $\Gamma_x$ and the parallel gradient $\nabla_s^\parallel$, as above. Implicit in the definition of the Wiener measure is a $\sigma$-algebra $\Sigma$ of measurable subsets of $PM$ together with a filtration $\Sigma_t\subset \Sigma$, which simply describes events which are observable until time $t$, i.e. which depend only on the $[0,t]$-part of the curves. A \emph{martingale} on $P_xM$ is a $\Sigma_t$-adapted integrable stochastic process $F_t:P_xM\to\mathbb{R}$ such that
\begin{equation}
F_{t_1} = E[F_{t_2}\, |\, \Sigma_{t_1}] \equiv E_{t_1}[F_{t_2}]\qquad (t_1\leq t_2).
\end{equation}
Here, the right hand side denotes the conditional expectation value on $P_xM$ given the $\sigma$-algebra $\Sigma_{t_1}$.  
The simplest examples of martingales on path space have the form
\begin{equation}\label{e:one_cylinder_example}
F_t(\gamma)=
\begin{cases}
    H_{T-t}f(\gamma(t)),& \text{if } t< T\\
    f(\gamma(T)),              & \text{if } t\geq T,
\end{cases}
\end{equation}
where $f:M\to\mathbb{R}$ and $T$ are fixed, and thus are indeed given by the (backwards) heat flow on $M$.  Given $F\in L^2(P_x M,\Gamma_x)$ we will often consider the induced martingale $F_t\equiv E_t[F]$.  From the above one might hypothesize that $E_t[F]$ plays a role similar to that of the (backwards) heat flow on $M$.  In fact, this analogy will develop much further as we progress.\\

\noindent\textbf{Evolution equations on path space.}
We found that the correct generalization of the Bochner inequality \eqref{bochner_inequ} on $M$ is given by a certain evolution inequality for martingales on $PM$. To get there, we start with by reformulating the martingale representation theorem and the Clark-Ocone formula \cite{Fang,Hsu} in the following way (see Section \ref{ss:martingale_representation} for a proof):

\begin{theorem}[Martingale Representation Theorem]\label{t:martingale_representation_intro}
If $F_t$ is a martingale on $P_xM$, then $F_t$ solves the stochastic differential equation
\begin{align}
dF_t = \langle \nabla_t^\parallel F_t, dW_t\rangle,
\end{align}
where $\nabla_t^\p$ is the parallel gradient (provided that $F_t$ is in the domain of $\nabla_t^\p$).
\end{theorem}

Note that the gradient and the expectation in Theorem \ref{t:martingale_representation_intro} are taken in the opposite order as in the usual formulation of the Clark-Ocone formula (this essentially amounts to a partial integration on path space). Expressed this way, we can view the martingale equation as an evolution equation on path space.  It is worth pointing out that the $dW_t$ term also behaves as a spatial derivative, in fact a form of divergence, so that the evolution equation in Theorem \ref{t:martingale_representation_intro} is analogous to a heat equation.

We then proceed by computing various evolution equations for associated quantities on path space. The most important for us is the following evolution equation for the parallel gradient of a martingale on path space.

\begin{theorem}[Evolution of the parallel gradient]\label{thm_ev_par_intro}
If $F_t:P_xM\to \dR$ is a martingale on path space, and $s\geq 0$ is fixed, then its $s$-parallel gradient $\nabla_s^\p F_t: P_xM\to T_xM$ satisfies the stochastic equation
\begin{align}
d \nabla_s^\p F_t = \langle \nabla_t^\p\nabla_s^\p F_t,dW_t\rangle + \frac{1}{2}\Ric_t(\nabla_t^\p F_t)\, dt + \nabla_{s}^\p F_s\,\delta_{s}(t)dt\, ,
\end{align}
where $\langle\Ric_t(X),Y\rangle=\Ric(P_t^{-1}X,P_t^{-1} Y)$ and $P_t=P_t(\gamma):T_{\gamma(t)}M\to T_xM$ is stochastic parallel transport.

\end{theorem}

Using Theorem \ref{thm_ev_par_intro} we can derive other evolution equations. In particular, we obtain our generalized Bochner formula:

\begin{theorem}[Bochner formula on path space]\label{thm_bochner_pathspace}
If $F_t:P_xM\to \dR$ is a martingale, and $s\geq 0$ is fixed, then
\begin{align}\label{equ_gen_boch}
d |\nabla_s^\parallel F_t|^2 = \langle \nabla_t^\p|\nabla_s^\p F_t|^2,dW_t\rangle + |\nabla_t^\p\nabla^\p_s F_t|^2dt +\Ric_t\big(\nabla_t^\p F_t,\nabla_s^\p F_t\big)\, dt + |\nabla_s^\p F_s|^2 \delta_{s}(t)dt\, ,
\end{align}
where $\Ric_t(X,Y)=\Ric(P_t^{-1}X,P_t^{-1} Y)$ and $P_t=P_t(\gamma):T_{\gamma(t)}M\to T_xM$ denotes stochastic parallel transport.
\end{theorem}

Theorem \ref{thm_bochner_pathspace} is the correct way to generalize the Bochner formula to path space. The crucial difference to the classical Bochner formula is that in the generalized Bochner formula \eqref{equ_gen_boch} the Ricci curvature, due to the nonpointwise nature of the $\nabla_s^\p$-gradient, enters in a more substantial way.
As a consequence, we will see that estimates derived from our generalized Bochner inequality (see Section \ref{ss:intro:bochner_ineq}) are actually strong enough to characterize two-sided Ricci bounds, and not just lower bounds.

Using our formalism, we can also compute many other useful evolution equations on path space (besides the ones from Theorem \ref{t:martingale_representation_intro}, Theorem \ref{thm_ev_par_intro} and Theorem \ref{thm_bochner_pathspace}); these additional formulas are in Section \ref{sec_ev_on_pathspace}.

\subsection{Generalized Bochner Inequality for Martingales}\label{ss:intro:bochner_ineq}

Using Theorem \ref{thm_bochner_pathspace} we then see that under the assumption of bounded Ricci curvature $|\Ric|\leq \kappa$ we have the generalized Bochner inequality
\begin{align}\label{e:intro:generalized_bochner_ineq}
	d |\nabla_s^\parallel F_t|^2 \geq \langle \nabla_t^\p|\nabla_s^\p F_t|^2,dW_t\rangle + |\nabla_t^\p\nabla^\p_s F_t|^2dt  - \kappa |\nabla_t^\p F_t|\,|\nabla_s^\p F_t|\, dt + |\nabla_s^\p F_s|^2 \delta_{s}(t)dt\, .
\end{align}

In the same vein as the classical case, from this one can formulate the dimensional generalized Bochner inequality
\begin{align}\label{e:intro:dim_bochner_ineq}
	d |\nabla_s^\parallel F_t|^2 \geq \langle \nabla_t^\p|\nabla_s^\p F_t|^2,dW_t\rangle + \tfrac{1}{n}|\Delta^\p_{s,t} F_t|^2dt  - \kappa |\nabla_t^\p F_t|\,|\nabla_s^\p F_t|\, dt + |\nabla_s^\p F_s|^2 \delta_{s}(t)dt\, ,
\end{align}
as well as the weak generalized Bochner formula
\begin{align}\label{e:intro:weak_bochner_ineq}
	d |\nabla_s^\parallel F_t|^2 \geq \langle \nabla_t^\p|\nabla_s^\p F_t|^2,dW_t\rangle  - \kappa |\nabla_t^\p F_t|\,|\nabla_s^\p F_t|\, dt + |\nabla_s^\p F_s|^2 \delta_{s}(t)dt\, .
\end{align}

We will see in Theorem \ref{thm_new_char} that these inequalities are in fact equivalent to the two sided Ricci curvature bound.  Additionally, we will see in the same way that the classical Bochner formula may be used to prove various gradient and hessian estimates on the heat flow on $M$, we can use the martingale Bochner formula to prove analogous estimates on martingales.  

To provide some brief intuition for the formula and its equivalence to a two sided Ricci bound, let us see that it genuinely generalizes the classical Bochner inequality.  That is, by applying \eqref{e:intro:generalized_bochner_ineq} for $s=0$ to the simplest functions on path space, namely those of the form $F(\gamma)\equiv f(\gamma(T))$, let us outline how we recover the classical Bochner inequality \eqref{bochner_inequ}:  Using \eqref{e:one_cylinder_example} and that $\nabla_0^\p$ is obtained by considering variations which are parallel it is an easy but instructive exercise to compute for $0\leq t\leq T$ that
\begin{align}\label{e:gen_boch_2_class_boch}
|\nabla_0^\p F_t|(\gamma) = |\nabla_t^\p F_t|(\gamma) = |\nabla H_{T-t} f|(\gamma(t))\, ,\;\;\text{ and } \;\; |\nabla_s^\p\nabla_t^\p F_t|(\gamma) = |\nabla^2 H_{T-t}f|(\gamma(t))\, .
\end{align}
Thus, the generalized Bochner inequality \eqref{e:intro:generalized_bochner_ineq} tells us that the process $X_t \equiv |\nabla H_{T-t}f|^2(\gamma(t))$ satisfies the evolution inequality
\begin{equation}\label{eq_intro_compare1}
dX_t-\langle \nabla_t^\p X_t,dW_t\rangle \geq  |\nabla^2 H_{T-t}f|^2\, dt - \kappa |\nabla H_{T-t} f|^2\, dt\, .
\end{equation}
On the other hand, applying the Ito formula to the process $X_t \equiv |\nabla H_{T-t}f|^2(\gamma(t))$ gives us that
\begin{align}\label{eq_intro_compare2}
d X_t - \langle \nabla_t^\p X_t,dW_t\rangle = \left(\tfrac{1}{2}\Delta + \partial_t\right)|\nabla H_{T-t}f|^2 \,dt\, .	
\end{align}
Comparing \eqref{eq_intro_compare1} with \eqref{eq_intro_compare2} we conclude that for each $f:M\to \dR$ we have
\begin{align}
	\left(\tfrac{1}{2}\Delta + \partial_t\right)|\nabla H_{T-t}f|^2 \geq |\nabla^2 H_{T-t} f|^2-\kappa |\nabla H_{T-t} f|^2\, ,
\end{align}
which is the backward time version of the classical Bochner inequality \eqref{bochner_inequ}.  In particular, this tells us that the martingale Bochner inequality \eqref{e:intro:generalized_bochner_ineq} implies that the Ricci curvature is bounded below by $-\kappa$. 
That the martingale Bochner inequality \eqref{e:intro:generalized_bochner_ineq} also captures the upper Ricci bound is a bit more subtle, and requires, roughly speaking, test functions where $\nabla_0^\p F_t\approx - \nabla_t^\p F_t$.
This will be made precise in Section \ref{ss:converse}.

\subsection{Applications of Martingale Bochner Formula}
We will now discuss four applications of our calculus for martingales on path space.\\

\noindent\textbf{New Characterizations of Bounded Ricci Curvature.}  Our first application is to give new characterizations of bounded Ricci curvature in terms of generalized Bochner inequalities on path space:

\begin{theorem}[New characterizations of bounded Ricci]\label{thm_new_char}
For a smooth complete Riemannian manifold $(M^n,g)$ the following are equivalent to the Ricci curvature bound $-\kappa g\leq \Rc \leq \kappa g$:
\begin{enumerate}[(C1)]
\item Martingales on path space satisfy the full Bochner inequality
\begin{align}\label{bochner_ineq_intro}
d |\nabla_s^\parallel F_t|^2 &\geq \langle \nabla_t^\p|\nabla_s^\p F_t|^2,dW_t\rangle +\abs{\nabla_t^\p\nabla_s^\p F_t}^2 dt -\kappa\abs{\nabla_s^\p F_t}\abs{\nabla_t^\p F_t}\, dt + \abs{\nabla_s^\p F_s}^2 \delta_s(t)dt\, .
\end{align}
\item Martingales on path space satisfy the dimensional Bochner inequality
\begin{align}\label{bochner_ineq_intro_dim}
d |\nabla_s^\parallel F_t|^2 &\geq \langle \nabla_t^\p|\nabla_s^\p F_t|^2,dW_t\rangle +\tfrac{1}{n}\abs{\Delta_{s,t}^\p F_t}^2 dt -\kappa\abs{\nabla_s^\p F_t}\abs{\nabla_t^\p F_t}\, dt + \abs{\nabla_s^\p F_s}^2 \delta_s(t)dt\, ,
\end{align}
where $\Delta_{s,t}^\p = \tr(\nabla_t^\p\nabla_s^\p)$ denotes the parallel Laplacian.
\item Martingales on path space satisfy the weak Bochner inequality
\begin{align}\label{bochner_ineq_weak_intro}
d |\nabla_s^\parallel F_t|^2 \geq \langle \nabla_t^\p|\nabla_s^\p F_t|^2,dW_t\rangle -\kappa \abs{\nabla_s^\p F_t} \abs{\nabla_t^\p F_t}\, dt + \abs{\nabla_s^\p F_s}^2 \delta_s(t)dt\, .
\end{align}
\item Martingales on path space satisfy the linear Bochner inequality
\begin{align}\label{bochner_ineq_weak_linear_intro}
d |\nabla_s^\parallel F_t| \geq \langle \nabla_t^\p|\nabla_s^\p F_t|,dW_t\rangle -\frac{\kappa}{2}\abs{\nabla_t^\p F_t}\, dt + \abs{\nabla_s^\p F_s} \delta_s(t)dt\, .
\end{align}
\item If $F_t$ is a martingale, then $t\mapsto |\nabla_s^\p F_t|+\frac{\kappa}{2}\int_s^t\abs{\nabla^\p_r F_r}\, dr$ is a submartingale for every $s\geq 0$.
\end{enumerate}
\end{theorem}

The estimates (C1) -- (C4) generalize the classical Bochner inequalities \eqref{bochner_inequ} -- \eqref{weak_bochner}, and the estimate (C5) generalizes that $e^{-\frac{\kappa}{2}t}|\nabla H_t f|$ is a subsolution to the heat flow. An interesting feature of (C2) is that while being an estimate on the infinite dimensional path space $PM$, it also captures the dimension $n$ of the manifold $M$. In stark contrast to the basic estimates \eqref{bochner_inequ} -- \eqref{weak_bochner}, our new estimates (C1) -- (C5) of Theorem \ref{thm_new_char} are strong enough to characterize two-sided Ricci bounds, and not just lower bounds. Additionally, we shall see below that the characterizations of Theorem \ref{thm_new_char}  give a new and vastly simplified proof of the previous characterizations of bounded Ricci curvature from \cite{Naber_char}.\\

\vspace{.25cm}

\noindent\textbf{New Gradient estimates for Martingales.}  The second application of our generalized Bochner formula concerns gradient estimates for martingales on the path space of manifolds with bounded Ricci curvature.

\begin{theorem}[{Gradient estimates for martingales}]\label{thm_improved_grad_ricci}
For a smooth complete Riemannian manifold $(M,g)$ the following are equivalent to the Ricci curvature bound $-\kappa g\leq \Rc \leq \kappa g$:
\begin{enumerate}[({G}1)]
\item For any $F\in L^2(PM)$ the induced martingale satisfies the gradient estimate
\begin{equation}
\abs{\nabla_s^\p F_t} \leq E_t \left[ \abs{\nabla_s^\p F}+\frac{\kappa}{2} \int_t^\infty e^{\tfrac{\kappa}{2}(r-t)} \abs{\nabla_r^\p F}\, dr \right]\, .
\end{equation}
\item For any $F\in L^2(PM)$ which is $\Sigma_T$-measurable the induced martingale satisfies the gradient estimate
\begin{equation}
\abs{\nabla_s^\p F_t}^2 \leq e^{\tfrac{\kappa}{2}(T-t)} E_t \left[ \abs{\nabla_s^\p F}^2+\frac{\kappa}{2} \int_t^T e^{\tfrac{\kappa}{2}(r-t)} \abs{\nabla_r^\p F}^2\, dr \right]\, .
\end{equation}
\end{enumerate}
\end{theorem}

Theorem \ref{thm_improved_grad_ricci} gives pointwise estimates for martingales on the path space of manifolds with bounded Ricci curvature.  These estimates generalize the heat flow estimate for spaces with lower Ricci curvature bounds given in \eqref{heat_grad_est}.  We will see these generalize the gradient estimates from \cite{Naber_char} as well. In fact, our estimates again characterize bounded Ricci curvature, i.e. the estimates (G1) and (G2) hold if and only if $\abs{\Ric}\leq \kappa$.  \\

\vspace{.25cm}


\noindent\textbf{New Hessian Estimates for Martingales.}  Our third application concerns new Hessian bounds for martingales on the path space of manifolds with bounded Ricci curvature. Morally, the Hessian term in the Bochner formula can be either simply discarded noticing that it has the good sign, or can be exploited more carefully. In the case of lower Ricci curvature the extra information contained in the Hessian term has been exploited quite deeply, e.g. in the proof of the splitting theorem \cite{CheegerGromoll,Gigli_splitting} and its effective versions \cite{CheegerColding_almost,CoNa_tc}. In the context of bounded Ricci curvature, we obtain the following new Hessian estimates for martingales on path space, estimates which are new even on $\dR^n$:

\begin{theorem}[Hessian Estimates]\label{thm_hess_est}
Let $(M,g)$ be a complete manifold with $\abs{\Ric}\leq \kappa$, and let $F\in L^2(P_xM)$ be $\Sigma_T$-measurable. Then it holds:
\begin{enumerate}[({H}1)]
\item  For each $s\geq 0$ we have the estimate $$\int_{PM} |\nabla_s^\p F_s|^2\, d\Gamma_x+\int_0^T\!\!\int_{PM}\ |\nabla_t^\p\nabla_s^\p F_t|^2\, d\Gamma_x\, dt \leq e^{\tfrac{\kappa}{2}(T-s)}  \int_{PM} \left(\abs{\nabla_s^\p F}^2+ \frac{\kappa}{2}\int_s^T e^{\tfrac{\kappa}{2}(t-s)}\abs{\nabla_t^\p F}^2\, dt \right) \, d\Gamma_x\, .$$
\item We have the Poincare Hessian estimate $$\int_{PM}\left(F-\int_{PM} F\, d\Gamma_x\right)^2  d\Gamma_x+ \int_0^T\!\!\!\int_0^T\!\!\int_{PM} |\nabla_t^\p\nabla_s^\p F_t|^2\, d\Gamma_x\, ds\, dt \leq e^{\tfrac{\kappa}{2}T} \int_0^T\!\!\! \int_{PM} \cosh(\tfrac{\kappa}{2}s)\abs{\nabla_s^\p F}^2 \, d\Gamma_x \, d s\, . $$ 
\item We have the log-Sobolev Hessian estimate
\begin{align}\int_{PM} F^2\ln F^2 \, d\Gamma_x &-\int_{PM} F^2 \, d\Gamma_x\, \ln \int_{PM} F^2 \, d\Gamma_x
+ \notag\\+&\frac12 \int_0^T\!\!\!\int_0^T\!\!\!\int_{PM} (F^2)_t |\nabla_t^\p\nabla_s^\p \ln(F^2)_t|^2 \, d\Gamma_x\, ds\, dt \leq 2e^{\tfrac{\kappa}{2}T} \int_0^T\!\!\! \int_{PM} \cosh(\tfrac{\kappa}{2}s)\abs{\nabla_s^\p F}^2 \, d\Gamma_x \, d s\notag\, .
\end{align}
\end{enumerate}
\end{theorem}

The estimates in Theorem \ref{thm_hess_est} can again be viewed as generalization for martingales on path space of some much more basic estimates for the  heat flow on $M$. For illustration, if $\kappa=0$ then the first estimate $(H1)$ combined with Doob's inequality for the submartingale $t\mapsto \abs{\nabla_s^\p F_t}$ gives the estimate
\begin{equation}
\sup_{t\geq 0} \int_{PM}\abs{\nabla_s^\p F_t}^2 \, d\Gamma_x +\int_0^\infty\!\!\!\int_{PM} \abs{\nabla_t^\p\nabla_s^\p F_t}^2 \, d\Gamma_x \, dt \leq 4 \int_{PM} \abs{\nabla_s^\p F}^2\, d\Gamma_x\, 
\end{equation}
for martingales on $PM$. This generalizes the classical $L^\infty H^1\cap L^2H^2$ estimate for the heat flow on $M$.\\

\noindent\textbf{New Proofs of the Characterizations of \cite{Naber_char}.}  In fact, although it will be apparent that the gradient and hessian estimates of the previous theorems generalize the estimates of \cite{Naber_char}, it is worth pointing out that the methods of this paper provide a new and streamlined proof of the characterizations of bounded Ricci curvature from \cite{Naber_char}:

\begin{theorem}[{Characterizations of bounded Ricci curvature \cite{Naber_char}}]\label{thm_char_ricci}
For a smooth complete Riemannian manifold $(M,g)$ the following are equivalent:
\begin{enumerate}[({R}1)]
\item The Ricci curvature satisfies the bound
\begin{equation}
-\kappa g\leq \Rc \leq \kappa g.
\end{equation}
\item For any $F\in L^2(PM)$ on total path space $PM$ we have the gradient estimate
\begin{equation}
\left|\nabla_x \int_{PM} F\, d\Gamma_x\right| \leq \int_{PM}\left(\abs{\nabla_0^\parallel F}+\int_0^\infty\frac{\kappa}{2}e^{\kappa s/2}\abs{\nabla_s^\parallel F}\, ds\right) \, d\Gamma_x.
\end{equation}
\item For any $F\in L^2(PM)$ on total path space $PM$ which is $\Sigma_T$-measurable we have the gradient estimate
\begin{equation}
\left|\nabla_x \int_{PM} F\, d\Gamma_x\right|^2 \leq e^{\tfrac{\kappa}{2}T} \int_{PM}\left(\abs{\nabla_0^\parallel F}^2+\int_0^T\frac{\kappa}{2}e^{\kappa s/2}\abs{\nabla_s^\parallel F}^2\, ds\right) \, d\Gamma_x.
\end{equation}

\item For any $F\in L^2(PM,\Gamma_x)$ on based path space $P_xM$, the quadratic variation of its induced martingale satisfies the estimate
\begin{equation}
\left| \int_{PM} \sqrt{\frac{d[F,F]_t}{dt}}  \, d\Gamma_x\right| \leq \int_{PM}\left(\abs{\nabla_t^\parallel F}+\int_t^\infty\frac{\kappa}{2}e^{\kappa (s-t)/2}\abs{\nabla_s^\parallel F}\, ds\right) \, d\Gamma_x.
\end{equation}

\item For any $F\in L^2(PM,\Gamma_x)$ on based path space $P_xM$ which is $\Sigma_T$-measurable the quadratic variation of its induced martingale satisfies the estimate
\begin{equation}
\left| \int_{PM} \frac{d[F,F]_t}{dt} \, d\Gamma_x\right| \leq e^{\tfrac{\kappa}{2}(T-t)} \int_{PM}\left(\abs{\nabla_t^\parallel F}^2+\int_t^\infty\frac{\kappa}{2}e^{\kappa (s-t)/2}\abs{\nabla_s^\parallel F}^2\, ds\right) \, d\Gamma_x.
\end{equation}

\item For any $F\in L^2(PM,\Gamma_x)$ on based path space $P_xM$ which is $\Sigma_T$-measurable, the twisted Ornstein-Uhlenbeck operator satisfies the spectral gap estimate
\begin{equation}
\int_{P_xM} \abs{F_{t_1}-F_{t_0}}^2  \, d\Gamma_x \leq e^{\tfrac{\kappa}{2}(T-t_0)}\int_{P_xM}\langle F, \mathcal{L}_{t_0,\kappa}^{t_1} F\rangle \, d\Gamma_x.
\end{equation}
\item For any $F\in L^2(PM,\Gamma_x)$ on based path space $P_xM$ which is $\Sigma_T$-measurable, the twisted Ornstein-Uhlenbeck operator satisfies the log-Sobolev inequality
\begin{equation}
\int_{P_xM} \abs{F^2}_{t_1}\log \abs{F^2}_{t_1}  \, d\Gamma_x - \int_{P_xM} \abs{F^2}_{t_0}\log \abs{F^2}_{t_0}  \, d\Gamma_x \leq 2e^{\tfrac{\kappa}{2}(T-t_0)}\int_{P_xM}\langle F, \mathcal{L}_{t_0,\kappa}^{t_1} F\rangle \, d\Gamma_x.
\end{equation}
\end{enumerate}
\end{theorem}

In the statement of (R6) and (R7) the twisted Ornstein-Uhlenbeck operator $\mathcal{L}_{t_0,\kappa}^{t_1}$ $(0\leq t_0\leq t_1\leq T\leq \infty)$ is defined by
\begin{multline}
\int_{P_xM}\langle F, \mathcal{L}_{t_0,\kappa}^{t_1} F\rangle \, d\Gamma_x\\
=\int_{P_xM}\left(\int_{t_0}^{t_1} \cosh(\tfrac{\kappa}{2}(s-t_0))\abs{\nabla_s^\p F}^2 ds+\frac{1-e^{-\kappa(t_1-t_0)}}{2}\int_{t_1}^\infty e^{\tfrac{\kappa}{2}(s-t_1)}\abs{\nabla_s^\p F}^2 ds\right)\, d\Gamma_x\, .
\end{multline}
In particular, $\mathcal{L}_{0,0}^\infty=\nabla^{\H \ast} \nabla^\H$ is the classical Ornstein-Uhlenbeck operator given by the composition of the Malliavin gradient and its adjoint.

Our new proof of Theorem \ref{thm_char_ricci} is very short, and vividly illustrates the efficiency of our martingale calculus. For illustration, if $\Rc=0$ then by the generalized Bochner inequality (C1) the process $t\mapsto\abs{\nabla^\p_s F_t}^2$ is a submartingale. Thus, by the very definition of a submartingale we get
\begin{equation}
\abs{\nabla^\p_s F_t}^2\leq E_t\left[\abs{\nabla^\p_s F_T}^2\right]\qquad (t\leq T).
\end{equation}
Taking the limit $T\to \infty$, and specializing to $s=t=0$, this implies the infinite dimensional gradient estimate (R3):
\begin{equation}\
\left|\nabla_x \int_{PM} F\, d\Gamma_x\right|^2 \leq \int_{PM}\abs{\nabla_0^\parallel F}^2 \, d\Gamma_x.
\end{equation}
The other estimates and estimates for nonzero $\kappa$ can be proven with similar ease.\\

\begin{remark}
With minor adjustments the results and proofs in this paper generalize to the case of smooth metric measure spaces $(M,g,e^{-f}dV_g)$ with $\abs{\Ric+\nabla^2 f}\leq \kappa$. However, for clarity of exposition we focus on the case of Riemannian manifolds with bounded Ricci.
\end{remark}

\begin{remark}
The methods introduced in the present paper can also be adapted for the time-dependent setting, and thus provide a useful tool for the study of Ricci flow using the framework from \cite{HaslhoferNaber_flow}. This will be discussed elsewhere.
\end{remark}

This article is organized as follows: In Section \ref{sec_prelim}, we discuss some preliminaries from stochastic analysis on manifolds. In Section \ref{sec_mart}, we discuss our interpretation of the martingale representation theorem (Theorem \ref{t:martingale_representation_intro}) and some of its consequences. In Section \ref{sec_ev_on_pathspace}, we derive all the relevant evolution equations on path space, in particular the evolution equation for the parallel gradient of martingales (Theorem \ref{thm_ev_par_intro}) and the generalized Bochner formula (Theorem \ref{thm_bochner_pathspace}). In Section \ref{sec_app1}, we discuss the four applications of our calculus on path space, i.e. we prove Theorem \ref{thm_new_char}, Theorem \ref{thm_improved_grad_ricci}, Theorem \ref{thm_hess_est} and Theorem \ref{thm_char_ricci}.

\section{Preliminaries}\label{sec_prelim}

\subsection{Frame bundle}
Given a complete Riemannian manifold $M$, let $\pi:FM\to M$ be the $O_n$-bundle of orthonormal frames. By definition, the fiber over a point $x\in M$ is given by the orthonormal maps $u:\mathbb{R}^n\to T_xM$. Thus, if $e_1,\ldots,e_n$ denotes the standard basis of $\dR^n$ then $ue_1,\ldots, ue_n$ is an orthonormal basis of $T_{x}M$, where $x=\pi(u)$.

A horizontal lift of a curve $x_t$ in $M$ is a curve $u_t$ in $FM$, with $\pi u_t=x_t$ and $\nabla_{\dot{x}_t}(u_t e_i)=0$ for $i=1,\ldots, n$. Once the initial point is specified, the horizontal lift exists and is unique. In particular, to each tangent vector $X\in T_{x}M$ we can associate a horizontal lift $X^\ast\in T_uFM$, for $u\in \pi^{-1}(x)$.

Given a representation $\rho$ of $O_n$ on a vector space $V$ and an equivariant map from $FM$ to $V$, we get a section of the associated vector bundle $FM\times_\rho V$, and vice versa. For example, a function $f:M\to \dR$ corresponds to the invariant function $\tilde{f}=f\pi:FM\to \dR$, and a vector field $Y\in\Gamma(TM)$ corresponds to a function $\tilde{Y}: F\to\dR^n$ via $\tilde{Y}(u)=u^{-1}Y_{\pi u}$, which is equivariant in the sense that $\tilde{Y}(ug)=g^{-1}\tilde{Y}(u)$. Covariant derivatives of tensors $T\in \Gamma(T^p_q M)$ can be expressed as horizontal derivatives of these equivariant functions, i.e.
\begin{equation}\label{hor_der_equiv}
\widetilde{\nabla_XT}=X^\ast \tilde{T},
\end{equation}
see e.g. \cite{KobNom}. On the frame bundle we have $n$ fundamental horizontal vector fields, defined by $H_i(u)=(ue_i)^\ast$. Using the fundamental horizontal vector fields, we can define the horizontal Laplacian $\Lap_{H}=\sum_{i=1}^n H_i^2$. As a consequence of \eqref{hor_der_equiv} we have
\begin{equation}
\widetilde{\Lap T}=\Lap_H \tilde{T},
\end{equation}
where $\Lap=g^{ij}\nabla_i\nabla_j$ is the Laplace-Beltrami operator on $M$, see e.g. \cite{KobNom}.

Besides the fundamental horizontal vector fields, we also have $n(n-1)/2$ fundamental vertical vector fields, defined by $V_{ij}(u)=\tfrac{d}{dt}|_{t=0} ue^{tA_{ij}}$, where $A_{ij}\in\mathfrak{o}(n)$ is the matrix whose $(i,j)$-th entry is $-1$, whose $(j,i)$-th entry is $+1$, and all whose other entries are zero. The following proposition gives the commutators between the fundamental vector fields.

\begin{proposition}[{see e.g. \cite{Hamilton_Harnack}}]\label{lemma_commutators}
The fundamental vector fields on the frame bundle satisfy the following commutator identities:
\begin{align}
[H_i, H_j]&=\tfrac12 R_{ijkl}V_{kl}\, ,\label{comm_hh}\\
[V_{ij}, H_k]&=\delta_{ik}H_{j}-\delta_{jk}H_i\, ,\\
[V_{ij},  V_{kl}]&=\delta_{ik}V_{jl}+\delta_{jl}V_{ik}-\delta_{il}V_{jk}-\delta_{jk}V_{il}\, ,
\end{align}
where $R_{ijkl}=\Rm(ue_i,ue_j,ue_k,ue_l)$.
\end{proposition}

Using Lemma \ref{lemma_commutators} we can easily compute all other relevant commutators, in particular we obtain:

\begin{corollary}\label{cor_commutators}
If $\tilde{f}:FM\to \mathbb{R}$ is an $O_n$-invariant function, then
\begin{align}
H_i H_j \tilde{f}-H_j H_i\tilde{f}&=0\, ,\\
\Lap_H H_i \tilde{f}-H_i\Lap_H \tilde{f}&=R_{ij} H_j\tilde{f},
\end{align}
where $R_{ij}=\Rc(ue_i,ue_j)$.
\end{corollary}

\begin{proof}
Since $\tilde{f}$ constant along fibres, we see that $V_{kl}\tilde{f}=0$, and the first formula follows from \eqref{comm_hh}. Using this, we compute
\begin{equation}
\Lap_H H_i\tilde{f}-H_i\Lap_H \tilde{f}=H_j H_i H_j \tilde{f}-H_i H_jH_j \tilde{f}= \tfrac{1}{2}R_{jikl}V_{kl} H_j \tilde{f}= \tfrac{1}{2}R_{jikl}(\delta_{kj}H_l-\delta_{lj}H_k) \tilde{f}=R_{ij}H_j\tilde{f}\, ,
\end{equation}
which proves the second formula.
\end{proof}

\subsection{Brownian motion and stochastic parallel transport} 

Brownian motion and stochastic parallel transport is most conveniently described via the Eells-Elworthy-Malliavin formalism. We give a quick summary here, and refer to \cite{Hsu} for a more gentle introduction.

Let $(P_0\mathbb{R}^n,\Sigma,\Gamma_0)$ be the the space of continuous curves in $\mathbb{R}^n$ starting at the origin, equipped with the Euclidean Wiener measure, and denote the Brownian motion map by $W_t:P_0\mathbb{R}^n\to \mathbb{R}^n$.

Given a point $x\in M$ and a frame $u$ above $x$, consider the following SDE on the frame bundle:
\begin{equation}
dU_t=\sum_{i=1}^n H_i(U_t) \circ dW_t^i,\qquad U_0=u.
\end{equation}
Then $X_t=\pi(U_t)$ is Brownian motion on $M$ starting at $x$, and $P_t=U_0U_t^{-1}: T_{X_t}M\to T_{x}M$ is a family of isometries, called stochastic parallel transport.
On the frame bundle, the Ito formula takes the form
\begin{equation}\label{ito_formula}
d\tilde{f}(U_t)=H_i \tilde{f} dW_t^i+\tfrac12 \Lap_H \tilde{f} dt.
\end{equation}
Note that the solution of the SDE defines maps $U:P_0\mathbb{R}^n\to P_u FM$ and $X:P_0\mathbb{R}^n\to P_x M$. The Wiener measure $\Gamma_x$ on $P_x M$ is then given as pushforward $\Gamma_x=X_\ast\Gamma_0$. More explicitly, the Wiener measure $\Gamma_x$ can be characterized as follows: If $e_{t_1,\ldots,t_N}:P_x M\to M^N$ denotes the evaluation map at the times $0\leq t_1 < \ldots < t_N$, then the pushforward of $\Gamma_x$ is given by the following product of heat kernel measures:
\begin{equation}
(e_{t_1,\ldots,t_N})_\ast d\Gamma_x (y_1,\ldots,y_N)=\rho_{t_1}(x,dy_1)\rho_{t_2-t_1}(y_1,dy_2)\cdots\rho_{t_N-t_{N-1}}(y_{N-1},dy_N).
\end{equation}
When there is no risk of confusion, we denote the $\sigma$-algebra $\Sigma$ on $P_0\mathbb{R}^n$ and $X_\ast \Sigma$ on $P_xM$ by the same letter, and we identify the isomorphic probability spaces $(P_0\mathbb{R}^n,\Sigma,\Gamma_0)$  and $(P_xM,\Sigma,\Gamma_x)$.  The $\sigma$-algebra comes with a natural filtration $\Sigma_t$ generated by the evaluation maps $e_{t'}$ with $t'\leq t$.

\begin{remark}
All our estimates imply a lower bound for the Ricci curvature. Thus, in our setting the assumption of metric completeness is equivalent to stochastic completeness.
\end{remark}

\subsection{Conditional expectation and martingales}

Let $F\in L^1(P_xM,\Gamma_x)$. We write $E[F]=\int_{P_xM} F d\Gamma_x$ for the expectation value of $F$. More generally, given $t\geq 0$ we write $F_t = E_t[F]\equiv E[F\, |\, \Sigma_t]$ for the conditional expectation of $F$ given the $\sigma$-algebra $\Sigma_t$, i.e. $F_t$ is the unique $\Sigma_t$-measurable function such that $\int_\Omega F_t \, d\Gamma_x = \int_\Omega F\,  d\Gamma_x$ for all $\Sigma_t$-measurable sets $\Omega$. Explicitly, $F_t$ is given by the formula
\begin{equation}
F_t(\gamma)=\int_{P_{\gamma(t)}M} F(\gamma|_{[0,t]}\ast \gamma')\, d\Gamma_{\gamma(t)}(\gamma'),
\end{equation}
where the integration is over all curves $\gamma'$ based at $\gamma(t)$, and where $\ast$ denotes concatenation.

We recall from the introduction, that a martingale on $P_xM$ is a $\Sigma_t$-adapted integrable stochastic process $F_t:P_xM\to\mathbb{R}$ such that
\begin{equation}
F_{t_1} = E_{t_1}[F_{t_2}] \qquad (t_1\leq t_2).
\end{equation}
Martingales on $P_xM$ are always continuous in time (possibly after modifying them on a set of measure zero, which we always tacitly assume).

By the definition of the conditional expectation, $F_t=E_t[F]$ is a martingale. Conversely, given any martingale $F_t: P_xM\to \mathbb{R}$ which is uniformly integrable, i.e. such that
\begin{equation}\label{uniform_integrability}
\limsup_{K\to \infty}\sup_{t>0}\int_{\{\gamma\in P_x M : \abs{F_t}(\gamma)> K\} } \abs{F_t}(\gamma)\, dP_x(\gamma) = 0,
\end{equation}
then by Doob's martingale convergence theorem we can take a limit $F_t\to F\in L^1(P_xM,\Gamma_x)$ as $t\to\infty$. In particular, each uniformly integrable martingale $F_t: P_xM\to \mathbb{R}$ can be represented in the form $F_t=E_t[F]$ for some $F\in L^1(P_xM,\Gamma_x)$.

\begin{example}
Let $f:M\to \dR$ be a smooth function with compact support and let $T>0$. Consider the function $F:P_xM\to \mathbb{R}$ defined by $F(\gamma)=f(\gamma(T))$. Then the induced martingale $F_t=E_t[F]$ is given by
\begin{equation}
F_t(\gamma)=
\begin{cases}
    H_{T-t}f(\gamma(t)),& \text{if } t< T\\
   f(\gamma(T)),              & \text{if } t\geq T,
\end{cases}
\end{equation}
where $H$ denotes the heat flow.
\end{example}

\begin{example}\label{ex_stopping_time}
Let $F_t:P_xM\to \dR$ be a martingale, and let $\tau:P_xM\to \dR^+$ be a stopping time, i.e. $\{\tau\leq t\}$ is $\Sigma_t$-measurable for each $t$.  Then the process $F_{t\wedge \tau}$ is a martingale. 
\end{example}

\begin{example}
Let $F_t:P_xM\to \dR$ be an $L^2$-martingale, and let $[F,F]_t$ be its quadratic variation.  Then the process $F^2_t-[F,F]_t$ is a martingale.
\end{example}

\subsection{Cylinder functions and approximation arguments}

A cylinder function $F:PM\to \dR$ is a function of the form
\begin{align}
F(\gamma) = f(\gamma(t_1),\ldots,\gamma(t_N))\, ,
\end{align}
where $f:M^N\to \dR$ is a smooth function with compact support and $0\leq t_1<\cdots<t_N<\infty$ is a partition. Cylinder functions are dense in $L^p$. Thus, to prove theorems on path space it often suffices to carry out the computations for cylinder functions, and then appeal to density. More precisely, the martingales in Theorem \ref{thm_new_char}, Theorem \ref{thm_hess_est} and Theorem \ref{thm_char_ricci} are of the form $F_t=E_t[F]$ where $F$ is in $L^2$, and thus can be approximated by cylinder functions (if $F$ is not in the domain of $\nabla_s^\p$, then $\abs{\nabla_s^\p F}=+\infty$ by convention, and any estimate where the right hand side is $+\infty$ holds trivially). In the theorems and propositions concerning evolution equations or evolution inequalities the martingale $F_t$ under consideration might violate the uniform integrability condition \eqref{uniform_integrability}. Nevertheless, for any $T<\infty$ we can still approximate $F_T$ by cylinder functions. We can then use this approximation by cylinder functions to prove the evolution formula on $[0,T]$, and then conclude that the evolution formulas hold in general, since $T$ was arbitrary.


\subsection{Parallel Gradient and Malliavin gradient}\label{ss:prelim:gradients}

Let $F:P_xM\to \dR$ be a cylinder function and let $s\geq 0$.  For $s\geq 0$ the $s$-parallel gradients are the one parameter family of gradients $\nabla_s^\p F: P_xM\to T_xM$ introduced in \cite{Naber_char} and defined by the formula
\begin{equation}
\langle \nabla_s^\p F(\gamma), Y\rangle = D_{Y_s}F(\gamma),
\end{equation}
where $Y_s(t)$ is the vector field along $\gamma(t)$ given by
\begin{align}
Y_s(t) \equiv 
\begin{cases}
0& \text{ if } t< s\, ,\notag\\
P^{-1}_tY& \text{ if } t\geq s\, ,
\end{cases}
\end{align}
and $P_t=P_t(\gamma):T_{\gamma(t)}M\to T_{x}M$ denotes stochastic parallel transport. That is, $\nabla_s^\p F$ is determined by variations of $F$ along the finite dimensional collection of curves which are parallel past the time $s$.  The $s$-parallel gradient is well defined for cylinder functions and may be extended as a closed linear operator on $L^2$ with the cylinder functions being a dense subset of the domain, see \cite{Naber_char}.  Explicitly, if $F(\gamma)=f(\gamma(t_1),\ldots, \gamma(t_N))$ is a cylinder function, then its $s$-parallel gradient can be computed via the formula
\begin{equation}
 \nabla_s^\p F(\gamma)=\sum_{t_\alpha\geq s} P_{t_\alpha}\nabla^{(\alpha)} f(\gamma(t_1),\ldots, \gamma(t_N)),
\end{equation}
where $\nabla^{(\alpha)}$ denotes the derivative with respect to the $\alpha$-th entry, and $P_{t_\alpha}=P_{t_\alpha}(\gamma):T_{\gamma(t_\alpha)}M\to T_{x}M$.\\

\begin{remark}
Note that $t\to \nabla_s^\p F_t$ is left continuous and thus a predictable process.
\end{remark}
\vspace{.5 cm}

In another direction, let $\H$ be the Hilbert-space of $H^1$-curves $\{y_t\}_{t\geq 0}$ in $T_xM$ with $y_0 = 0$, equipped with the inner product
\begin{equation}
\langle y,z\rangle_{\H}=\int_0^\infty \langle \tfrac{d}{dt}{y}_t,\tfrac{d}{dt}{z}_t\rangle\, dt.
\end{equation}
If $F:P_xM\to \dR$ is a cylinder function then its Malliavin gradient is the unique almost everywhere defined function $\nabla^\H F : P_xM \to \H$, such that
\begin{equation}
D_Y F(\gamma) = \langle \nabla^\H F(\gamma),v\rangle_{\H}
\end{equation}
for every $v\in\mathcal{H}$ for almost every Brownian curve $\gamma$, where $Y= \{P_t^{-1}y_t \}_{t\geq 0}$.

Explicitly, if $F(\gamma)=f(\gamma(t_1),\ldots,\gamma(t_N))$, then
\begin{equation}
\sum_{\alpha=1}^N \langle y_{t_\alpha},P_{t_\alpha}\nabla^{(\alpha)}f\rangle = D_Y F = \langle \nabla^\H F,y\rangle_{\H}=\int_0^\infty \langle \tfrac{d}{dt} \nabla^\H F, \tfrac{d}{dt} y_t\rangle\, dt.
\end{equation}
It follows that
\begin{equation}
\tfrac{d}{dt} \nabla^\H F = \sum_{\alpha=1}^N 1_{\{t\leq t_\alpha\}}P_{t_\alpha}\nabla^{(\alpha)}f=\nabla_t^\p F,
\end{equation}
i.e. the parallel gradient is the derivative of the Malliavin gradient. In particular, we have the formula
\begin{equation}\label{eq_mall_grad}
\abs{\nabla^\H F}_{\H}^2=\int_0^\infty \abs{\nabla_s^\p F}^2\, ds.
\end{equation}

As above, having defined the Malliavin gradient in the special case of cylinder functions, it can be extended to closed unbounded operator on $L^2$, with the cylinder functions as a dense subset of its domain.



\section{A reinterpretation of martingale formulas}\label{sec_mart}

The formulas of this section are all classical in nature, but rewritten in a way which will be particularly natural in our context and will reinforce the interpretation of martingales as a form of (backwards) heat flow.  These interpretations will play an important role in subsequent sections.\\

\subsection{Martingale representation theorem}\label{ss:martingale_representation}

Let us begin with the martingale representation formula, which tells us that every martingale $F_t$ is the Ito integral of some stochastic process with respect to Brownian motion.  More precisely,
\begin{align}
dF_t =\, <X_t, dW_t>\, ,
\end{align}
for some predictable stochastic process $X_t$.  There have been several results, in particular the Clark-Ocone theorem \cite{Fang,Hsu}, which give methods for computing $X_t$.  However, our first goal in this section is to see how to compute $X_t$ directly from $F_t$ itself.  In this way we will be able to view the martingale equation as an evolution equation on path space.

\begin{theorem}[Martingale representation theorem]\label{t:martingale_representation}
If $F_t$ is a martingale on $P_xM$, and $F_t$ is in the domain of $\nabla_t^\p$, then $F_t$ solves the stochastic differential equation
\begin{align}
dF_t = \langle \nabla_t^\parallel F_t, dW_t\rangle. 
\end{align}
\end{theorem}

\begin{proof}
Let $f:M^N\to \mathbb{R}$ be a smooth function with compact support. Let $F:PM\to \mathbb{R}$ be the function
\begin{equation}
F(X)=f(X_{t_1},\ldots,X_{t_N}).
\end{equation}
Consider the lift $\tilde{f}:FM^N\to \mathbb{R}$, $\tilde{f}(u_1,\ldots,u_N)=f(\pi u_1,\ldots \pi u_N)$. Let $PFM$ be the path space of the frame bundle and consider $\tilde{F}: PFM\to \mathbb{R}$, $\tilde{F}(U)=\tilde{f}(U_{t_1},\ldots, U_{t_N})$.

Let $F_t=E_t[F]$ be the martingale induced by $F$ and assume $t\in(t_\beta,t_{\beta+1})$. Then
\begin{align}\label{e:martingale_cylinder}
F_t(X)&=\int_{M^{N-\beta}}f(X_{t_1},\ldots, X_{t_\beta},y_{\beta+1},\ldots, y_N) \rho_{t_{\beta+1}-t}(X_t,dy_{\beta+1})\cdots \rho_{t_N-t_{N-1}}(y_{N-1},dy_N)\nonumber\\
&=: f_t(X_{t_1},\ldots, X_{t_\beta},X_t).
\end{align}
Note that the function $(t,x)\mapsto f_t(x_1,\ldots,x_\beta,x)$ is uniformly Lipschitz in the time variable and satisfies
\begin{equation}
(\partial_t+\tfrac12 \Lap^{(\beta+1)})f_t=0,
\end{equation}
where the Laplacian acts on the last variable. The lift of $F_t$ to the frame bundle is given by
\begin{equation}
\tilde{F}_t(U)=\tilde{f}_t(U_{t_1},\ldots,U_{t_\beta},U_t).
\end{equation}
Using the Ito formula \eqref{ito_formula} we compute
\begin{equation}
d\tilde{F}_t(U)= \langle H^{(\beta+1)}\tilde{f}_t,dW_t\rangle+(\partial_t+\tfrac12 \Lap_H^{(\beta+1)})\tilde{f}_tdt=\langle H^{(\beta+1)}\tilde{f}_t,dW_t\rangle,
\end{equation}
where the horizontal derivative and the horizontal Laplacian act on the last variable.
Projecting down to $M$ this implies the martingale representation formula:
\begin{equation}
dF_t=\langle \nabla_t^\p F_t,dW_t\rangle.
\end{equation}
Indeed, the projected equation can be obtained by computing
\begin{equation}
\langle\nabla_t^\p F_t,dW_t\rangle\equiv(U_0^{-1}\nabla_t^\p F_t)^i dW^i_t=(U_t^{-1}\nabla f_t|_{X_t})^i dW^i_t,
\end{equation}
and
\begin{equation}
(U_t^{-1}\nabla f_t|_{X_t})^i =\langle \nabla f_t|_{X_t},U_te_i\rangle_{T_{X_t}M}=(U_t e_i)f_t|_{X_t}=H_i \tilde{f}_t|_{U_t}.
\end{equation}
This proves the martingale representation theorem for cylinder functions, and thus by density for all functions in the domain of the parallel gradient.
\end{proof}

An interesting corollary is the following:

\begin{corollary}\label{cor_quadr_var}
Let $F_t$ be an $L^2$-martingale on $P_xM$. Then the quadratic variation $[F,F]_t$ of $F_t$ satisfies
\begin{align}
d[F,F]_t = |\nabla_t^\p F_t|^2 dt\, .
\end{align}
\end{corollary}

An equally interesting corollary is the following:

\begin{corollary}
Let $F_t$ be an Ito process on $P_xM$, such that $F_t$ is in the domain of $\nabla_t^\p$. Then the quadratic variation term $[F,W]_t$ is given by
\begin{align}
d[F,W]_t =U_0^{-1}\nabla_t^\p F_t\,dt\, .
\end{align}
\end{corollary}

Most importantly, the representation formula of Theorem \ref{t:martingale_representation} leads to the following corollary, which can be viewed as a representation theorem for submartingales.

\begin{corollary}[Submartingale representation theorem]\label{t:submartingale_representation}
Let $F_t$ be an Ito process on $P_xM$, such that $F_t$ is in the domain of $\nabla_t^\p$. Then $F_t$ is a submartingale if and only if it satisfies the stochastic differential inequality
\begin{align}
dF_t \geq \langle \nabla_t^\p F_t, dW_t\rangle. 
\end{align}
\end{corollary}

Though basic, the above formula will be important to us as it will allow us to easily identify and exploit submartingales from their evolution equations in a manner mimicking the finite dimensional context.

\subsection{Ito formula and Ito isometry}

From the point of view adopted in Theorem \ref{t:martingale_representation}, we may rewrite the Ito formula in the following manner:

\begin{theorem}[Ito formula]\label{t:ito_formula}
Let $F_t$ be a martingale on $P_xM$, such that $F_t$ is in the domain of $\nabla_t^\p$, and let $\phi:\dR\to\dR$ be a $C^2$-function. Then $\phi(F_t)$ solves the stochastic differential equation
\begin{align}\label{e:Ito_formula}
d\phi(F_t) = \langle \nabla_t^\p\, \phi(F_t), dW_t\rangle + \frac{1}{2}\phi''(F_t)|\nabla_t^\p F_t|^2\, dt. 
\end{align}
\end{theorem}

\begin{proof}
Using the standard Ito formula, Theorem \ref{t:martingale_representation} and Corollary \ref{cor_quadr_var} we compute
\begin{align}
d\phi(F_t)&=\phi'(F_t)\, dF_t + \tfrac{1}{2}\phi''(F_t)\, d[F,F]_t\\
&= \phi'(F_t) \langle \nabla_t^\p F_t, dW_t\rangle + \frac{1}{2}\phi''(F_t)\, |\nabla_t^\p F_t|^2 dt. 
\end{align}
Noticing also that $\nabla_t^\p \phi(F_t)= \phi'(F_t)\nabla_t^\p F_t$, this proves the assertion.
\end{proof}

\begin{remark}
Let us make the following comparison.  Assume $f_t:M\to \dR$ solves the backward heat equation $\partial_t f_t = -\frac{1}{2}\Delta f_t$ and that $\phi:\dR\to \dR$ is $C^2$-function. Then  $\phi(f_t)$ solves the equation
\begin{align}
\partial_t \phi(f_t) = -\frac{1}{2}\Delta \phi(f_t) + \frac{1}{2}\phi''(f_t)|\nabla f_t|^2\, .
\end{align}
\end{remark}
\begin{remark}
In particular, one can view \eqref{e:Ito_formula} as a generalization of Jensen's inequality for martingales on $P_xM$.  Indeed, if $\phi$ is a convex then combining \eqref{e:Ito_formula} with Corollary \ref{t:submartingale_representation} we have that $\phi(F_t)$ is a submartingale.
\end{remark}
\begin{remark}
More generally if $F_t, G_t$ are martingales and $\phi,\psi:\dR\to\dR$ are $C^2$-functions then
\begin{align}
d\Big(\phi(F_t)\psi(G_t)\Big) =& \langle \nabla_t^\p\Big(\phi(F)\psi(G)\Big),dW_t\rangle + \frac{1}{2}\phi''(F_t)|\nabla_t^\p F_t|^2\psi(G_t)\, dt + \frac{1}{2}\phi(F_t)\psi''(G_t)|\nabla_t^\p G_t|^2\, dt \notag\\
& + \langle \nabla_t^\p \phi(F_t), \nabla_t^\p \psi(G_t)\rangle\, dt\, .
\end{align}
\end{remark}

To finish this section, let us observe that from the point of view adapted in Theorem \ref{t:martingale_representation}, we may rewrite the Ito isometry in the following manner. 

\begin{theorem}[Ito isometry]\label{thm_ito_isom}
Let $F\in L^2(P_xM)$.  Then
\begin{align}
E\left[\int_0^\infty |\nabla_t^\p F_t|^2dt\right] = E\Big[(F-E[F])^2\Big]\, .
\end{align}
\end{theorem}

\begin{proof}
Using the classical Ito isometry and Theorem \ref{t:martingale_representation} we compute
\begin{equation}
E\left[\int_0^\infty |\nabla_t^\p F_t|^2dt\right]= E\left[\left( \int_0^\infty \langle \nabla_t^\p F_t,dW_t\rangle \right)^2\right] 
 =E\left[\left( \int_0^\infty dF_t \right)^2\right] =E\left[\left(F-E[F] \right)^2\right] \, .
\end{equation}
This proves the assertion.
\end{proof}
\vspace{.5cm}

\section{Evolution equations on path space and Generalized Bochner Formula}\label{sec_ev_on_pathspace}

When doing analysis on $M$ one considers a solution $f_t$ of the heat flow, and then computes the evolution equations of quantities associated to $f_t$.  In this spirit, the goal of this section is to compute the evolution equations for various quantities associated to martingales on path space, such as its square, its parallel gradient, its Malliavin gradient, etc. In particular, we will prove our generalized Bochner formula. In this section we tacitly assume that the martingales are sufficiently regular, i.e. in the domains of the respective parallel gradients.  

\begin{proposition}\label{prop_zero_ev} 
If $F_t:P_xM\to \dR$ is a martingale on path space, then the following hold:
\begin{enumerate}[(1)]
\item $dF_t^2 = \langle \nabla_t^\p F_t^2,dW_t\rangle+|\nabla_t^\p F_t|^2 dt$,
\item $d|F_t| = \langle \nabla_t^\p |F_t|,dW_t\rangle+\abs{\nabla_t^\p F_t}^2 dL_t$, where $L_t=\lim_{\eps\to 0}\tfrac{1}{2\eps}\abs{\{s\in [0,t]\, |\,  F_t\in (-\eps,\eps)\}}$.
\end{enumerate}
\end{proposition}

\begin{proof}
By the martingale representation theorem (Theorem \ref{t:martingale_representation}) we have the evolution equation
\begin{equation}
dF_t  = \langle \nabla_t^\p F_t, dW_t\rangle. 
\end{equation}
We can thus apply the Ito formula (Theorem \ref{t:ito_formula}) with $\phi(x)=\abs{x}^2$ to obtain
\begin{equation}
dF_t^2 = \langle \nabla_t^\p F_t^2,dW_t\rangle+|\nabla_t^\p F_t|^2 dt.
\end{equation}
This proves (1). Similarly, (2) follows by approximating $\phi(x)=\abs{x}$ by the $C^2$-functions
\begin{equation}
\phi_\eps(x)=(x^2/2\eps+\eps/2)1_{\abs{x}<\eps}+\abs{x}1_{\abs{x}\geq \eps}\, ,
\end{equation}
applying the Ito formula (Theorem \ref{t:ito_formula}), and taking the limit $\eps\to 0$.
\end{proof}

Next, and most importantly, we compute the evolution equation for the parallel gradient of a martingale:

\begin{theorem}[Evolution of the parallel gradient]\label{thm_evol_par_grad}
If $F_t:P_xM\to \dR$ is a martingale on path space, and $s\geq 0$ is fixed, then $\nabla_s^\p F_t: P_xM\to T_xM$ satisfies the stochastic differential equation
\begin{align}\label{eq_nablasft}
d \nabla_s^\p F_t = \langle \nabla_t^\p\nabla_s^\p F_t,dW_t\rangle + \frac{1}{2}\Ric_t(\nabla_t^\p F_t)\, dt + \nabla_{s}^\p F_s\,\delta_{s}(t)dt\, .
\end{align}
\end{theorem}

\begin{remark}\label{rem_delta_notation}
Since $F_t$ is $\Sigma_t$-measurable, the parallel gradient $\nabla_s^\parallel F_t$ is identically zero  for $t<s$. For $t>s$, we will show that $d \nabla_s^\p F_t$ satisfies the evolution equation $d \nabla_s^\p F_t = \langle \nabla_t^\p\nabla_s^\p F_t,dW_t\rangle + \frac{1}{2}\Ric_t(\nabla_t^\p F_t)\, dt$. Thus, using the $\delta$-notation, the evolution equation for $\nabla_s^\p F_t$ can be summarized in form \eqref{eq_nablasft} which is valid for any $t$.
\end{remark}

\begin{proof}[{Proof of Theorem \ref{thm_evol_par_grad}}]
Fix $s$, and consider $t> s$.  We will use freely the notation developed in the preliminaries of Section \ref{sec_prelim}. Consider $F(\gamma)=f(\gamma(t_1),\ldots,\gamma(t_N))$ and observe since $s$ is fixed that $\nabla_s^\p F_t$ is well behaved over the evaulation times, hence it is enough for us to consider the evolution equation for $t\in(t_\beta,t_{\beta+1})$.  Using the notation of \eqref{e:martingale_cylinder} we have
\begin{equation}
\nabla_s^\p F_t = \sum_{t_\alpha\geq s}P_{t_\alpha} \nabla^{(\alpha)} f_t (X_{t_1},\ldots,X_{t_\beta},X_t)+P_t\nabla^{(\beta+1)}f_t (X_{t_1},\ldots,X_{t_\beta},X_t),
\end{equation}
where $\nabla^{(\alpha)}$ acts on the $\alpha$-th entry.  On the frame bundle, this is represented by the functions
\begin{equation}
G_i(U):=\sum_{t_\alpha\geq s}H_{i}^{(\alpha)} \tilde{f}_t (U_{t_1},\ldots, U_{t_\beta}, U_t)
+H_{i}^{(\beta+1)} \tilde{f}_t (U_{t_1},\ldots, U_{t_\beta}, U_t),
\end{equation}
where the horizontal vector field $H^{(\alpha)}$ acts on the $\alpha$-th entry. Using the Ito formula \eqref{ito_formula}, we compute
\begin{align}
d G_i(U)&=\sum_{t_\alpha\geq s}H_j^{(\beta+1)}H_{i}^{(\alpha)} \tilde{f}_t (U_{t_1},\ldots, U_{t_\beta}, U_t)dW_t^j+H_j^{(\beta+1)}H_{i}^{(\beta+1)} \tilde{f}_t (U_{t_1},\ldots, U_{t_\beta}, U_t)dW_t^j\\
&\quad +\sum_{t_\alpha\geq s} (\partial_t+\tfrac12\Lap_H^{(\beta+1)})H_{i}^{(\alpha)} \tilde{f}_t (U_{t_1},\ldots, U_{t_\beta}, U_t)dt
+(\partial_t+\tfrac12 \Lap_H^{(\beta+1)})H_{i}^{(\beta+1)} \tilde{f}_t (U_{t_1},\ldots, U_{t_\beta}, U_t)dt\nonumber\\
&=H_j^{(\beta+1)}\left(\sum_{t_\alpha\geq s}H_{i}^{(\alpha)} \tilde{f}_t (U_{t_1},\ldots, U_{t_\beta}, U_t)+H_{i}^{(\beta+1)} \tilde{f}_t (U_{t_1},\ldots, U_{t_\beta}, U_t)\right)dW_t^j\\
&\quad+\tfrac12 R_{ij}H_{j}^{(\beta+1)} \tilde{f}_t (U_{t_1},\ldots, U_{t_\beta}, U_t)dt\, ,\nonumber
\end{align}
where in the last step we used Corollary \ref{cor_commutators} and the equation $(\partial_t+\tfrac12 \Lap_H^{(\beta+1)})\tilde{f}_t=0$. Pushing down to $M$ this gives
\begin{equation}
d\nabla_s^\p F_t=\langle \nabla_t^\p\nabla_s^\p F_t,dW_t\rangle +\tfrac12 \Ric_t(\nabla_t^\p F_t) \, dt.
\end{equation}
Taking also into account Remark \ref{rem_delta_notation}, this proves the theorem.
\end{proof}
\vspace{.25cm}

Using Theorem \ref{thm_evol_par_grad} and the Ito formula, we can compute all other relevant evolution equations.

\begin{theorem}[Generalized Bochner Formula on $PM$]\label{cor_mart_first_order}
Let $F_t:P_xM\to \dR$ be a martingale.
\begin{enumerate}[(1)]
\item If $s\geq 0$ is fixed, then $\nabla_s^\p F_t: P_xM\to T_xM$ satisfies the following stochastic equations:
\begin{enumerate}
\item $d |\nabla_s^\p F_t|^2 = \langle \nabla^\p_t|\nabla_s^\p F_t|^2,dW_t\rangle + |\nabla^\p_t\nabla_s^\p F_t|^2dt +\Ric_t\big(\nabla_s^\p F_t,\nabla_t^\p F_t\big)\, dt + |\nabla_s^\p F_s|^2 \delta_{s}(t)dt$
\item $d |\nabla_s^\p F_t| = \langle \nabla_t^\p|\nabla_s^\p F_t|,dW_t\rangle + \frac{|\nabla_t^\p\nabla_s^\p F_t|^2-|\nabla^\p_t|\nabla^\p_s F||^2}{2|\nabla^\p_s F_t|}dt +\tfrac{1}{2|\nabla^\p_s F_t|} \Ric_t\big(\nabla^\p_s F_t,\nabla^\p_t F_t\big)\, dt + |\nabla^\p_s F_s| \delta_{s}(t)dt$
\end{enumerate}


\item The following stochastic equations hold:
\begin{enumerate}
\item $d |\nabla^\H F_t|^2 = \langle \nabla^\p_t |\nabla^{\H} F_t|^2,dW_t\rangle + \Big(\int_0^\infty \big(|\nabla_t^\p\nabla^\p_s F_t|^2 + \Ric_t(\nabla^\p_s F_t,\nabla^\p_t F_t)\big)\, ds+ |\nabla^\p_t F_t|^2 \Big)\,dt$
\item $d \int_0^\infty|\nabla^\p_s F_t| ds = \langle \nabla^\p_t \int_0^\infty |\nabla^\p_s F_t| ds,dW_t\rangle + \Big(\int_0^\infty \frac{|\nabla^\p_t\nabla^\p_s F_t|^2-|\nabla^\p_t|\nabla^\p_s F_t||^2 + \Ric_t(\nabla^\p_s F_t,\nabla^\p_t F_t)}{2\abs{\nabla^\p_s F_t}}\, ds+ |\nabla^\p_t F_t| \Big)\,dt$
\end{enumerate}
\end{enumerate}
\end{theorem}

\begin{proof}
We will use Theorem \ref{thm_evol_par_grad} and the Ito formula repeatedly.

Assume $t>s$. Note first that equation \eqref{eq_nablasft} implies that the quadratic variation $[\nabla_s^\p F_t,\nabla_s^\p F_t]$ satisfies
\begin{equation}\label{form_quad_var_par_grad}
d [\nabla_s^\p F_t,\nabla_s^\p F_t]=\abs{\nabla_t^\p\nabla_s^\p F_t}^2\, dt.
\end{equation}
Using this, the Ito formula, and equation \eqref{eq_nablasft} we compute
\begin{align}
d |\nabla_s^\p F_t|^2 &= 2\langle \nabla_s^\p F_t, d\nabla_s^\p F_t\rangle + d [\nabla_s^\p F_t,\nabla_s^\p F_t] \\
 &= \langle \nabla^\p_t|\nabla_s^\p F_t|^2,dW_t\rangle +\Ric_t\big(\nabla_t^\p F_t,\nabla_s^\p F_t\big)\, dt + |\nabla^\p_t\nabla_s^\p F_t|^2dt. 
\end{align}
Observing that $|\nabla_s^\p F_t|^2=0$ for $t<s$ implies the correct $\delta$-term. This proves (1a).

We continue by computing, assuming again $t>s$, that
\begin{equation}
d |\nabla_s^\p F_t|^2=2 |\nabla_s^\p F_t|\, d |\nabla_s^\p F_t|+ d[\abs{\nabla_s^\p F_t},\abs{\nabla_s^\p F_t}].
\end{equation}
Inserting the formula from (1a) and rearranging terms this implies
\begin{equation}
d |\nabla_s^\p F_t| = \langle \nabla^\p_t |\nabla_s^\p F_t| ,dW_t\rangle + \tfrac{1}{2|\nabla^\p_s F_t|}\left(|\nabla_t^\p\nabla_s^\p F_t|^2 + \Ric_t\big(\nabla^\p_t F_t,\nabla^\p_s F_t\big)\right) dt - \tfrac{1}{2|\nabla^\p_s F_t|}d[\abs{\nabla_s^\p F_t},\abs{\nabla_s^\p F_t}].
\end{equation}
Considering the coefficient in front of $dW_t$ we infer that
\begin{equation}
d[\abs{\nabla_s^\p F_t},\abs{\nabla_s^\p F_t}]=\abs{\nabla^\p_t |\nabla_s^\p F_t| }^2\, dt\, .
\end{equation}
Observing that $|\nabla_s^\p F_t|=0$ for $t<s$, one can again infer the correct $\delta$-term. Equation (1b) follows.\footnote{Note that, contrary to Proposition \ref{prop_zero_ev}, there no term $L_t$ capturing the local time spent at the origin since such a term only shows up in 1 dimension, but $\nabla_s^\p F_t$ is vector valued (we tacitly assume that $n>1$).}

Finally, using the formula
\begin{equation}
\abs{\nabla^\H F_t}^2=\int_0^\infty \abs{\nabla_s^\p F_t}^2 \, ds,
\end{equation}
equation (2a) follows from (1a). Similarly, (2b) follows from (1b).
\end{proof}
\vspace{.25cm}

As a corollary we can produce the following:

\begin{proposition}\label{cor_poincare_ev} If $F_t: P_xM\to\mathbb{R}$ is a martingale and $X_t\equiv |\nabla^\H F_t|^2 - F_t^2$ then
\begin{equation}
d X_t = \langle \nabla_t^\p X_t,dW_t\rangle + \Big(\int_0^\infty\big(|\nabla^\p_t\nabla^\p_s F_t|^2 +\Ric_t(\nabla^\p_t F_t,\nabla^\p_s F_t)\big)\, ds\Big)\,dt\, .
\end{equation}
\end{proposition}

\begin{proof}
This follows by combining part (1) of Proposition \ref{prop_zero_ev} and part (2a) of Proposition \ref{cor_mart_first_order}.
\end{proof}
\vspace{.25cm}

\begin{proposition}\label{cor_logsob_ev}
If $F_t: P_xM\to\mathbb{R}$ is a nonnegative martingale and $X_t\equiv F_t^{-1}\abs{\nabla^{\H}F_t}^2-2 F_t\ln F_t$, then
\begin{equation}
d X_t = \langle \nabla_t^\p X_t,dW_t\rangle + F_t\left(\int_{0}^\infty\big|\nabla^\p_t\nabla^\p_s \ln F_t\big|^2ds\right)\, dt+F_t^{-1}\left(\int_{0}^\infty \Ric_t\big(\nabla_s^\p F_t,\nabla_t^\p F_t\big)ds\right)\, dt\, .
\end{equation}
\end{proposition}

\begin{proof}
Using the Ito formula (Theorem \ref{t:ito_formula}) we start by computing
\begin{equation}\label{comb_1st}
d (F_t\ln F_t) = \langle\nabla^\p_t (F_t\ln F_t),dW_t\rangle + \tfrac{1}{2}F_t^{-1}\abs{\nabla^\p_t F_t}^2 dt,
\end{equation}
and
\begin{equation}\label{eq_finv}
d F_t^{-1} = \langle\nabla^\p_t F_t^{-1},dW_t\rangle + F_t^{-3}\abs{\nabla^\p_t  F_t}^2 dt.
\end{equation}
Next, using again the Ito formula, part (2a) of Proposition \ref{cor_mart_first_order}, and equation \eqref{eq_finv}, we compute
\begin{align}\label{eq_logsobproof1}
d (F_t^{-1}\abs{\nabla^{\H}F_t}^2) &=  F_t^{-1}  d \abs{\nabla^{\H}F_t}^2 + \abs{\nabla^{\H}F_t}^2 dF_t^{-1} + d[F_t^{-1},\abs{\nabla^{\H}F_t}^2]\nonumber\\
& = F_t^{-1}\langle \nabla_t^\p \abs{\nabla^\H F_t}^2,dW_t\rangle + F_t^{-1}\left( \int_0^\infty (\abs{\nabla_s^\p\nabla_t^\p F_t}^2+\Ric_t(\nabla_s^\p F_t, \nabla_t^\p F_t))\, ds+\abs{\nabla_t^\p F_t}^2\right)\, dt\nonumber\\
&\quad  + F_t^{-1}\langle \nabla_t^\p \abs{\nabla^\H F_t}^2,dW_t\rangle + F_t^{-3}  \abs{\nabla_t^\p F_t}^2 \abs{\nabla^\H F_t}^2 \, dt + \langle \nabla_t^\p F_t^{-1},\nabla_t^\p \abs{\nabla^\H F_t}^2\rangle\, dt.
\end{align}
The terms in this expression can be grouped together nicely, namely we have that
\begin{equation}\label{eq_logsobproof2}
F_t^{-1}\langle \nabla_t^\p \abs{\nabla^\H F_t}^2,dW_t\rangle + F_t^{-1}\langle \nabla_t^\p \abs{\nabla^\H F_t}^2,dW_t\rangle 
= \langle \nabla_t^\p ( F_t^{-1}\abs{\nabla^\H F_t}^2),dW_t\rangle \, ,
\end{equation}
and
\begin{align}\label{eq_logsobproof3}
&\!\!\! F_t^{-1}\left( \int_0^\infty \abs{\nabla_s^\p\nabla_t^\p F_t}^2\, ds\right)\, dt + F_t^{-3}  \abs{\nabla_t^\p F_t}^2 \abs{\nabla^\H F_t}^2\, dt + \langle \nabla_t^\p F_t^{-1},\nabla_t^\p \abs{\nabla^\H F_t}^2\rangle\, dt \nonumber\\
&\qquad\qquad = F_t^{-1}\left( \int_0^\infty \left( \abs{\nabla_s^\p\nabla_t^\p F_t}^2 + F_t^{-2}\abs{\nabla_t^\p F_t}^2 \abs{\nabla_s^\p F_t}^2 - 2F_t^{-1}\langle \nabla_s^\p\nabla_t^\p F_t,\nabla_s^\p F_t \otimes \nabla_t^\p F_t\rangle\right) \, ds\right)\, dt \nonumber\\
&\qquad\qquad = F_t \left( \int_0^\infty \abs{\nabla_s^\p\nabla_t^\p \ln F_t}^2\, ds\right)\, dt\, .
\end{align}
Putting together the equations \eqref{eq_logsobproof1}, \eqref{eq_logsobproof2} and \eqref{eq_logsobproof3}, we infer that
\begin{multline}\label{comb_2nd}
d (F_t^{-1}\abs{\nabla^{\H}F_t}^2) = \langle \nabla_t^\p ( F_t^{-1}\abs{\nabla^\H F_t}^2),dW_t\rangle + F_t \left( \int_0^\infty \abs{\nabla_s^\p\nabla_t^\p \ln F_t}^2\, ds\right)\, dt\\
+F_t^{-1}\left(\int_{0}^\infty \Ric_t\big(\nabla_s^\p F_t,\nabla_t^\p F_t\big)ds\right)\, dt + \abs{\nabla_t^\p F_t}^2 dt.
\end{multline}
Combining equation \eqref{comb_2nd} with equation \eqref{comb_1st} this proves the assertion.
\end{proof}
\vspace{.5cm}

\section{Characterizations of bounded Ricci and Hessian estimates}\label{sec_app1}

Using the formalism developed, we will now prove the main estimates and results of the paper.  The main theorems (Theorem \ref{thm_new_char}, Theorem \ref{thm_improved_grad_ricci}, Theorem \ref{thm_hess_est} and Theorem \ref{thm_char_ricci}) will be proved in tandem, in a manner designed to make the logical ordering as quick as possible. As a spin off of proving all the new estimates and characterizations, we will also see how to reproduce the estimates of \cite{Naber_char} through a vastly simplified procedure.

\subsection{Proof of $|\Ric|\leq \kappa\implies (C1)\implies (C2)\implies (C3)$}

Using the evolution equation for $|\nabla_s^\p F_t|^2$ in Theorem \ref{cor_mart_first_order} we see that the Ricci curvature bound $\abs{\Ric}\leq \kappa$ implies the claimed estimate $(C1)$:
\begin{equation}\label{ev_inequ_b2}
 d |\nabla_s^\p F_t|^2 \geq \langle \nabla_t^\p|\nabla_s^\p F_t|^2,dW_t\rangle + \abs{\nabla_s^\p\nabla_t^\p F_t}^2\, dt  -\kappa \abs{\nabla^\p_s F_t}\abs{\nabla^\p_t F_t}\, dt + \abs{\nabla_s^\p F_s}^2 \delta_s(t)dt\, .
\end{equation}

Now since $\nabla_s^\p\nabla_t^\p F_t\in T_x M\otimes T_xM$ is symmetric we have the pointwise inequality $|\nabla_s^\p\nabla_t^\p F_t|^2\geq \tfrac{1}{n}|\Delta_{s,t}^\p F|^2$, which immediately yields $(C2)$:
\begin{equation}\label{ev_inequ_b3}
 d |\nabla_s^\p F_t|^2 \geq \langle \nabla_t^\p|\nabla_s^\p F_t|^2,dW_t\rangle + \tfrac{1}{n}\abs{\Delta_{s,t}^\p F_t}^2\, dt  -\kappa \abs{\nabla^\p_s F_t}\abs{\nabla^\p_t F_t}\, dt + \abs{\nabla_s^\p F_s}^2 \delta_s(t)dt\, .
\end{equation}
Finally, dropping the nonnegative term $\tfrac{1}{n}\abs{\Delta_{s,t}^\p F_t}^2$ the dimensional Bochner inequality (C2) of course implies the weak Bochner inequality (C3).

\subsection{Proof of $(C1)\implies (C4)\Longleftrightarrow (C5)$}

Assume $t>s$. We start by expressing the left hand side of (C1) as
\begin{equation}
 d |\nabla_s^\p F_t|^2 = 2|\nabla_s^\p F_t| \,d |\nabla_s^\p F_t|+ d[ |\nabla_s^\p F_t|, |\nabla_s^\p F_t|]\, .
\end{equation}
Note that for the quadratic variation term we have
\begin{equation}
 d[ |\nabla_s^\p F_t|, |\nabla_s^\p F_t|]=\abs{\nabla_t^\p \abs{\nabla_s^\p F_t}}^2 \, dt\leq \abs{\nabla_t^\p \nabla_s^\p F_t}^2 \, dt\, .
\end{equation}
Combining these facts, we see that (C1) implies
\begin{equation}
 d |\nabla_s^\p F_t| \geq \langle \nabla_t^\p|\nabla_s^\p F_t|,dW_t\rangle -\frac{\kappa}{2} \abs{\nabla^\p_t F_t}\, dt \, .
\end{equation}
Together with the fact that $\nabla_s^\p F_t=0$ for $t<s$ this yields the linear Bochner formula $(C4)$:
\begin{equation}
 d |\nabla_s^\p F_t| \geq \langle \nabla_t^\p|\nabla_s^\p F_t|,dW_t\rangle -\frac{\kappa}{2} \abs{\nabla^\p_t F_t}\, dt + \abs{\nabla_s^\p F_s}\, \delta_s(t)dt\, .
\end{equation}

In order to conclude $(C5)$ observe that we can write (C4) in the form
\begin{equation}
 d\left( |\nabla_s^\p F_t| +\frac{\kappa}{2}\int_s^t \abs{\nabla^\p_r F_r}\, dr\right) \geq \left\langle \nabla_t^\p\left( |\nabla_s^\p F_t| +\frac{\kappa}{2}\int_s^t \abs{\nabla^\p_s F_s}\, ds\right),dW_t\right\rangle  + |\nabla_s^\p F_s|\,\delta_s(t)dt\, .
\end{equation}
Thus, using the representation theorem for submartingales (Corollary \ref{t:submartingale_representation}) and the fact that $\nabla_s^\p F_t=0$ for $t<s$, we see that (C4) and (C5) are equivalent.

\subsection{Proof of $(C5)\implies (G1)$}

Using (C5) and the defining property of submartingales we obtain
\begin{equation}\label{e:req_plugin}
|\nabla_s^\p F_t| \leq  E_t\left[\abs{\nabla_s^\p F}\right]+ E_t\left[ \frac{\kappa}{2}  \int_t^\infty \abs{\nabla^\p_{r} F_{r}}\, dr\right]\, .
\end{equation}
We may estimate $|\nabla^\p_r F_r|$ by applying the above with $s=t=r$ to infer
\begin{equation}
|\nabla_s^\p F_t| \leq  E_t\left[\abs{\nabla_s^\p F}\right]
+ E_t\left[ \frac{\kappa}{2} \int_t^\infty \abs{\nabla^\p_{t_1} F} \, dt_1  + \left(\frac{\kappa}{2}\right)^2  \int_t^\infty \int_{t_1}^\infty  \abs{\nabla^\p_{t_2} F_{t_2}} \, dt_2dt_1\right]\, ,
\end{equation}
where we have used that $E_t\left[E_{t_1}[\cdot]\right]=E_t[\cdot]$ if $t\leq t_1$.  Plugging in equation \eqref{e:req_plugin} recursively we arrive at
\begin{equation}
|\nabla_s^\p F_t| \leq  E_t\left[\abs{\nabla_s^\p F} + \sum_{j=1}^\infty \left(\frac{\kappa}{2}\right)^j \int_t^\infty \cdots \int_{t_{j-1}}^\infty \abs{\nabla^\p_{t_j} F}\, dt_j\ldots dt_1\right]\, .
\end{equation}
Noticing that the volume of the simplex $\{(t_1,\ldots,t_{j-1})\,  |\, t\leq t_1\leq \ldots \leq t_{j-1}\leq t_j\}$ equals $(t_j-t)^{j-1}/(j-1)!$ we get
\begin{align}
\sum_{j=1}^\infty\left(\frac{\kappa}{2}\right)^j \int_t^\infty \cdots \int_{t_{j-1}}^\infty \abs{\nabla^\p_{t_j} F}\, dt_j\ldots dt_1  &= \sum_{j=1}^\infty  \left(\frac{\kappa}{2}\right)^j  \int_t^\infty \frac{(t_j-t)^{j-1}}{(j-1)!}\abs{\nabla^\p_{t_j} F}\, dt_j\nonumber\\
& = \frac{\kappa}{2} \int_t^\infty e^{\tfrac{\kappa}{2}(r-t)}\abs{\nabla^\p_r F}\, dr\, .
\end{align}
Putting things together, this proves the gradient estimate $(G1)$:
\begin{equation}
\abs{\nabla_s^\p F_t} \leq E_t \left[ \abs{\nabla_s^\p F}+\frac{\kappa}{2} \int_t^\infty e^{\tfrac{\kappa}{2}(r-t)} \abs{\nabla_r^\p F}\, dr\right]\, .
\end{equation}

\subsection{Proof of $(G1)\implies (G2)$}

Let $F$ be $\Sigma_T$-measurable. Using the gradient estimate (G1) and H\"older's inequality we have
\begin{align}
\abs{\nabla_s^\p F_t}^2\leq E_t \left[ \left(\abs{\nabla_s^\p F}+ \int_t^T \frac{\kappa}{2}e^{\tfrac{\kappa}{2}(r-t)}\abs{\nabla_r^\p F}\, dr\right)^2 \right]\, .
\end{align}
Together with the inequality $(a+b)^2\leq \gamma a^2 + \tfrac{\gamma}{\gamma-1}b^2$ this implies
\begin{align}
\abs{\nabla_s^\p F_t}^2\leq  E_t \left[ e^{\tfrac{\kappa}{2}(T-t)}\abs{\nabla_s^\p F}^2+\frac{e^{\tfrac{\kappa}{2} (T-t)}}{e^{\tfrac{\kappa}{2} (T-t)}-1} \left(\int_t^T \frac{\kappa}{2}e^{\tfrac{\kappa}{2}(r-t)}\abs{\nabla_r^\p F}\, dr\right)^2 \right]\, .
\end{align}
Taking into account H\"older's inequality, which yields
\begin{equation}
\left(\int_t^T \frac{\kappa}{2}e^{\tfrac{\kappa}{2}(r-t)} \abs{\nabla_r^\p F}\, dr\right)^2\leq \left(e^{\frac{\kappa}{2} (T-t)}-1\right) \int_t^T \frac{\kappa}{2}e^{\tfrac{\kappa}{2}(r-t)} \abs{\nabla_r^\p F}^2\, dr\, ,
\end{equation}
we obtain the quadratic gradient estimate (G2):
\begin{equation}
\abs{\nabla_s^\p F_t}^2 \leq e^{\tfrac{\kappa}{2}(T-t)} E_t \left[ \abs{\nabla_s^\p F}^2+\frac{\kappa}{2} \int_t^T e^{\tfrac{\kappa}{2}(r-t)} \abs{\nabla_r^\p F}^2\, dr \right]\, .
\end{equation}

\subsection{Proof of $(H1)$}

Suppose $\abs{\Ric}\leq \kappa$ and let $F$ be $\Sigma_T$-measurable. Integrating $(C1)$ from $0$ to $T$ and taking the expectation value we obtain
	\begin{equation}\label{der_first_hess}
E\left[ \abs{\nabla_s^\p F_s}^2 \right]+ E\left[ \int_0^T \abs{\nabla_t^\p\nabla_s^\p F_t}^2 \, dt \right]   \leq E\left[ \abs{\nabla_s^\p F}^2 \right]+ \kappa E\left[ \int_0^T \abs{\nabla_s^\p F_t}\abs{\nabla_t^\p F_t} \, dt \right] \, .
\end{equation}

To proceed we need to estimate the last term.  The following claim provides the correct estimate:\\

{\bf Claim 1: } We have $\kappa E\left[ \int_0^{T} \abs{\nabla_s^\p F_t} \abs{\nabla_t^\p F_t}\, dt\right]\leq E\left[ \left(e^{\tfrac{\kappa}{2}(T-s)}-1\right)\abs{\nabla_s^\p F}^2+\frac{\kappa}{2} e^{\frac{\kappa}{2}T}\int_s^{T} e^{\tfrac{\kappa}{2}(r-2s)} \abs{\nabla_r^\p F}^2\, dr \right]\, .$\\

To prove Claim 1 we start by observing that
\begin{equation}\label{eq_comb1}
\kappa E\left[ \int_0^{T} \abs{\nabla_s^\p F_t} \abs{\nabla_t^\p F_t}\, dt\right]=
\kappa E\left[ \int_s^{T} \abs{\nabla_s^\p F_t} \abs{\nabla_t^\p F_t}\, dt\right]\leq 
\frac{\kappa}{2} E\left[ \int_s^{T} \left( \abs{\nabla_s^\p F_t}^2+ \abs{\nabla_t^\p F_t}^2\right)\, dt\right]\, .
\end{equation}
Using the gradient estimate (G2) we get
\begin{align}\label{eq_comb2}
\frac{\kappa}{2} E\left[ \int_s^{T} \abs{\nabla_s^\p F_t}^2\, dt\right]&\leq
\frac{\kappa}{2} E \left[\int_s^{T} e^{\tfrac{\kappa}{2}(T-t)}\left(\abs{\nabla_s^\p F}^2  +  \frac{\kappa}{2} \int_t^T e^{\tfrac{\kappa}{2}(r-t)}\abs{\nabla^\p_rF}^2\, dr \right)\, dt \right]\nonumber\\
&= 
\left(e^{\tfrac{\kappa}{2}(T-s)}-1\right) E \left[\abs{\nabla_s^\p F}^2  \right]+\left(\frac{\kappa}{2}\right)^2 e^{\frac{\kappa}{2}T}
E \left[\int_s^{T}\!\!\! \int_t^T e^{\tfrac{\kappa}{2}(r-2t)}\abs{\nabla^\p_rF}^2\, dr \, dt \right]\, .
\end{align}
Proceeding similarly, we get get the estimate
\begin{equation}\label{eq_comb3}
\frac{\kappa}{2} E\left[ \int_s^{T} \abs{\nabla_t^\p F_t}^2\, dt\right]\leq 
\frac{\kappa}{2} E \left[ \int_s^{T} e^{\tfrac{\kappa}{2}(T-r)}\abs{\nabla_r^\p F}^2 \, dr \right]
+\left(\frac{\kappa}{2}\right)^2 e^{\frac{\kappa}{2}T}
E \left[\int_s^{T}\!\!\! \int_t^T e^{\tfrac{\kappa}{2}(r-2t)}\abs{\nabla^\p_rF}^2\, dr \, dt \right]\, .
\end{equation}
Changing the order of integration we compute
\begin{align}\label{eq_comb4}
\kappa\int_s^{T}\!\!\! \int_t^T e^{\tfrac{\kappa}{2}(r-2t)}\abs{\nabla^\p_rF}^2\, dr \, dt 
=\kappa\int_s^{T}\!\!\! \int_s^r e^{\tfrac{\kappa}{2}(r-2t)}\abs{\nabla^\p_rF}^2\, dt \, dr  
=\int_s^{T}\left( e^{\tfrac{\kappa}{2}(r-2s)}-e^{-\tfrac{\kappa}{2}r} \right) \abs{\nabla^\p_rF}^2 \, dr \, .
\end{align}
Combining \eqref{eq_comb1}, \eqref{eq_comb2}, \eqref{eq_comb3} and \eqref{eq_comb4} the claim follows. $\square$ \\

Now if we plug in the estimate of Claim 1 into \eqref{der_first_hess} we immediately conclude $(H1)$.

\subsection{Proof of $(H1)\implies (H2)$}

To prove $(H2)$ we integrate $(H1)$ for $0\leq s\leq T$ and use that $E\left[(F-E[F])^2\right]=E\left[\int_0^T |\nabla_s^\p F_s|^2\, ds\right]$ by the Ito isometry (see Theorem \ref{thm_ito_isom}) in order to get the estimate
\begin{multline}\label{e:h1_h2:1}
E\left[\big(F-E[F]\big)^2\right] + E\left[\int_0^T\!\!\int_0^T |\nabla_t^\p\nabla_s^\p F_t|^2\, dt\, ds\right ]\\
\leq e^{\frac{\kappa}{2}T}E\left[\int_0^T e^{-\frac{\kappa}{2}s}|\nabla_s^\p F|^2\, ds+ \frac{\kappa}{2}\int_0^T\!\!\int_s^T e^{\frac{\kappa}{2}(t-2s)}|\nabla_t^\p F|^2\, dt\, ds\right]\, .	
\end{multline}

Switching the order of integration in the last term gives
\begin{align}\label{e:h1_h2:2}
\kappa \int_0^T\!\!\!\int_s^T e^{\frac{\kappa}{2}(t-2s)}|\nabla_t^\p F|^2dt\,ds = 
\kappa\int_0^T\!\!\!\int_0^t e^{\frac{\kappa}{2}(t-2s)}|\nabla_t^\p F|^2ds\,dt
= \int_0^T \big(e^{\frac{\kappa}{2}t}-e^{-\frac{\kappa}{2}t}\big)|\nabla_t^\p F|^2\, dt\, .
\end{align}

Combining this with \eqref{e:h1_h2:1} proves the Poincare Hessian estimate $(H2)$.

\subsection{Proof of $(H3)$}

To prove $(H3)$ we could proceed as in the proof of $(H2)$ by first finding an evolution equation for $F_t^{-1}|\nabla_s F_t|^2$ and proceeding in a manner analogous to $(H1)\implies (H2)$.  Instead, in an attempt to illustrate another method with the Bochner techniques, we will rely on Proposition \ref{cor_logsob_ev}, which provides an evolution equation involving the full $H^1$-gradient of $F_t$.  Thus let us consider $G\equiv F^2$ and apply Proposition \ref{cor_logsob_ev} to get the evolution inequality
\begin{align}
&\!\!\! d (G_t^{-1}\abs{\nabla^{\H}G_t}^2-2 G_t\log G_t) \\
&\geq \langle \nabla_t^\p (G_t^{-1}\abs{\nabla^{\H}G_t}^2-2 G_t\log G_t),dW_t\rangle +
G_t\left(\int_{0}^T\big|\nabla^\p_t\nabla^\p_s \log G_t\big|^2ds\right)\, dt-\kappa G_t^{-1}\left(\int_{0}^T \abs{\nabla_s^\p G_t} \, ds\right)\,  \abs{\nabla_t^\p G_t}\, dt \, .\nonumber
\end{align}
Integrating this from $0$ to $T$ and taking the expectation value, we obtain
\begin{multline}\label{eq_to_der_log_sob}
E \left[ G\log G\right] - E\left[G\right] \log E\left[G\right] +\frac{1}{2}E\left[\int_0^T\!\!\!\int_{0}^T G_t\big|\nabla^\p_t\nabla^\p_s \log G_t\big|^2ds\, dt\right] \\
\leq 
2E\left[ \abs{\nabla^\H F}^2+\frac{\kappa}{4}\int_0^T\!\!\!\int_0^T G_t^{-1} \abs{\nabla_s^\p G_t}\abs{\nabla_t^\p G_t}\, ds\, dt\right] \, .
\end{multline}

To proceed we need the following error estimate:\\

{\bf Claim 2: } We have $\frac{\kappa}{4}E\left[\int_0^T\!\int_0^T G_t^{-1}\abs{\nabla_s^\p G_t}\abs{\nabla_t^\p G_t}\, ds\, dt \right]\leq
E\left[ \int_0^T \left( e^{\tfrac{\kappa}{2}T} \cosh\left(\frac{\kappa}{2}s\right)-1\right) \abs{\nabla_s^\p F}^2 \, ds \right]$ .\\

To prove Claim 2 we start by observing
\begin{equation}\label{log_eq_comb1}
 E\left[ \int_0^{T} G_t^{-1}\abs{\nabla_s^\p G_t} \abs{\nabla_t^\p G_t}\, dt\right]=
 E\left[ \int_s^{T} G_t^{-1}\abs{\nabla_s^\p G_t} \abs{\nabla_t^\p G_t}\, dt\right]\leq 
\frac{1}{2} E\left[ \int_s^{T} G_t^{-1} \left( \abs{\nabla_s^\p G_t}^2+ \abs{\nabla_t^\p G_t}^2\right)\, dt\right]\, .
\end{equation}
Using the gradient estimate (G1) and the Cauchy-Schwarz inequality as in the proof of (G2) we get
\begin{align}
\abs{\nabla_s^\p G_t}^2 &\leq E_t\left [ 2F \left( \abs{\nabla_s^\p F}+\frac{\kappa}{2}\int_t^T e^{\tfrac{\kappa}{2}(r-t)}\abs{\nabla_r^\p F}\, dr\right) \right]^2\nonumber\\
&\leq 4 G_t \, e^{\tfrac{\kappa}{2}(T-t)}
E_t\left [  \abs{\nabla_s^\p F}^2+\frac{\kappa}{2}\int_t^T e^{\tfrac{\kappa}{2}(r-t)}\abs{\nabla_r^\p F}^2\, dr \right]\, .
\end{align}
This implies
\begin{align}\label{log_eq_comb2}
\frac{\kappa}{8} E\left[ \int_s^{T} G_t^{-1}\abs{\nabla_s^\p G_t}^2\, dt\right]\leq 
\left(e^{\tfrac{\kappa}{2}(T-s)}-1\right) E \left[\abs{\nabla_s^\p F}^2  \right]+\left(\frac{\kappa}{2}\right)^2 e^{\frac{\kappa}{2}T}
E \left[\int_s^{T}\!\!\! \int_t^T e^{\tfrac{\kappa}{2}(r-2t)}\abs{\nabla^\p_rF}^2\, dr \, dt \right]\, .
\end{align}
Proceeding similarly, we get get the estimate
\begin{align}\label{log_eq_comb3}
\frac{\kappa}{8} E\left[ \int_s^{T} G_t^{-1}\abs{\nabla_t^\p G_t}^2\, dt\right]\leq 
\frac{\kappa}{2} E \left[\int_s^T e^{\tfrac{\kappa}{2}(T-r)}\abs{\nabla_r^\p F}^2 \, dr \right]+\left(\frac{\kappa}{2}\right)^2 e^{\frac{\kappa}{2}T}
E \left[\int_s^{T}\!\!\! \int_t^T e^{\tfrac{\kappa}{2}(r-2t)}\abs{\nabla^\p_rF}^2\, dr \, dt \right]\, .
\end{align}
Combining \eqref{log_eq_comb1}, \eqref{log_eq_comb2}, \eqref{log_eq_comb3} and \eqref{eq_comb4} we obtain the error estimate
\begin{equation}\label{log_first_error_est}
\frac{\kappa}{4}E\left[ \int_0^{T} G_t^{-1}\abs{\nabla_s^\p G_t} \abs{\nabla_t^\p G_t}\, dt\right]\leq
\left(e^{\tfrac{\kappa}{2}(T-s)}-1\right) E\left[ \abs{\nabla_s^\p F}^2\right]
+\frac{\kappa}{2}e^{\tfrac{\kappa}{2}T} E\left[ \int_s^T e^{\tfrac{\kappa}{2}(r-2s)}\abs{\nabla_r^\p F}^2 \, dr\right]\, .
\end{equation}
Integrating \eqref{log_first_error_est} over $s$ from $0$ to $T$, and computing the double integral as in \eqref{e:h1_h2:2}, the claim follows. $\square$\\

Now combining \eqref{eq_to_der_log_sob} and Claim 2 we conclude that
\begin{multline}
E\left[ F^2\ln F^2 \right] -E[F^2] \log E[ F^2 ]
+ \frac12 E\left[ \int_0^T\!\!\!\int_0^T(F^2)_t |\nabla_t^\p\nabla_s^\p \ln(F^2)_t|^2 \, ds\, dt \right]\\
\leq 2e^{\tfrac{\kappa}{2}T} E\left[\int_0^T\cosh(\tfrac{\kappa}{2}s)\abs{\nabla_s^\p F}^2 \, d s \right] \, .
\end{multline}
This proves the log-Sobolev Hessian estimate (H3), and thus finishes the proof of Theorem \ref{thm_hess_est}.

\subsection{Proof of $(R2)-(R7)$}

We briefly remark how the estimates (R2) -- (R7), which are the estimates from \cite{Naber_char}, easily follow from our new estimates.  Indeed, $(R2)$ and $(R3)$ follow by evaluating $(G1)$ and $(G2)$, respectively, by evaluating at $s=t=0$.  The estimates $(R4)$ and $(R5)$ similarly follow from $(G1)$ and $(G2)$ by setting $s=t$, integrating both sides over $PM$, and recalling the equality $\frac{d[F,F]_t}{dt} = |\nabla_t^\p F_t|^2$  from Corollary \ref{cor_quadr_var}.

The estimate $(R6)$ is essentially a weaker form of $(H2)$ obtained by dropping the Hessian term.  Precisely, as in the proof of $(H2)$ if we integrate $(H1)$ for $t_0\leq s\leq t_1$ and drop the Hessian term then we arrive at the inequality
\begin{align}
E\left[|F_{t_1}-F_{t_0}|^2\right] \leq e^{\frac{\kappa}{2}T}E\left[\int_{t_0}^{t_1}e^{-\frac{\kappa}{2}s}|\nabla_s^\p F|^2\, ds+\frac{\kappa}{2}\int_{t_0}^{t_1}\int_{s}^Te^{\frac{\kappa}{2}(t-2s)}|\nabla_t^\p F|^2 \, dt \, ds \right]	\, .
\end{align}
Changing the order of integration for the second term and proceeding as in \eqref{e:h1_h2:2} finishes the proof of $(R6)$.  As with $(R6)$ the estimate $(R7)$ is essentially a weaker version of $(H3)$ obtained by dropping the Hessian term. The proof follows in verbatim the manner of $(H3)$, however we integrate Proposition \ref{cor_logsob_ev} from $t_0\leq t\leq t_1$, instead of over the whole interval $0\leq t\leq T$.

\subsection{Proof of Converse Implications}\label{ss:converse}

In order to finish the proof of the Theorem \ref{thm_new_char}, Theorem \ref{thm_improved_grad_ricci} and Theorem \ref{thm_char_ricci} we need to see the converse implications, namely that the desired estimates themselves imply the bounds on Ricci curvature.  We will split this into two parts, namely the proof of the lower bound and the proof of the upper bounds.  The verbatim test functions we will introduce may used to prove any of the converse implications, and so we will focus in this subsection on $(C3)\implies |\Ric|\leq \kappa$, which is to say we will see that the weak Bochner inequality implies the two sided Ricci curvature bound.\\

\noindent{\bf (C3) implies Lower Ricci.}  We saw in the introduction how the martingale Bochner inequality may be used to imply the classical Bochner inequality, and therefore the lower Ricci bound.  Regardless, it is instructive for us to prove directly the lower bound, as a slightly more involved version of the same technique will be used to prove the upper bound.  Thus for $x\in M$ and $v\in T_xM$ a unit vector let us choose a smooth compactly supported function $f_1:M\to \dR$ such that
\begin{align}\label{e:converse:0}
f_1(x)=0\, , \;\;\;\; \nabla f_1(x) = v\, ,\;\;\;\; \nabla^2 f_1(x)=0\, .	
\end{align}
Note one can build such a function by using exponential coordinates.  If we consider the function on path space given by $F_\epsilon(\gamma) = f_1(\gamma(\epsilon))$, then let us observe for $s\leq t\leq \epsilon$ the computations
\begin{align}\label{e:converse:1}
\nabla_t^\p F_t = P_{t}\nabla H_{\epsilon-t}f_1(\gamma(t))\, ,\;\;\;\; |\nabla_t^\p\nabla_s^\p F_t| = |\nabla^2 H_{\epsilon-t}f_1|(\gamma(t))\, .
\end{align}
Note in particular that $\nabla_t^\p F_t\approx v$ and $|\nabla_t^\p\nabla_s^\p F_t|\approx 0$ for $\epsilon$ small, at least for a typical curve (one can be very effective about this estimate, but it is not necessary for our purpose).  Now using the generalized Bochner formula of Theorem \ref{thm_bochner_pathspace} we have that $t\mapsto |\nabla_0^\p F_t|^2-\int_0^t \big(|\nabla_r^\p\nabla_s^\p F_r|^2 + \Ric(\nabla_0^\p F_r,\nabla_r^\p F_r)\, dr\big)$ is a martingale.  In particular we have
\begin{align}\label{e:converse:eq:1}
|\nabla^\p_0 F_0|^2=E\left[|\nabla_0^\p F_\epsilon|^2-\int_0^\epsilon \big(|\nabla_r^\p\nabla_s^\p F_r|^2 + \Ric(\nabla_0^\p F_r,\nabla_r^\p F_r)\, dr\big)\right]\, .
\end{align}
Now by using \eqref{e:converse:1} we get
\begin{align}\label{e:converse:2}
|\nabla^\p_0 F_0|^2=E\left[|\nabla_0^\p F_\epsilon|^2\right]-\epsilon Rc(v,v)+o(\epsilon)\, .
\end{align}

On the other hand, by $(C3)$ we have that $t\mapsto |\nabla_0^\p F_t|^2+\kappa\int_0^t |\nabla_0^\p F_r| |\nabla_r^\p F_r|\, dr$ is a submartingale, so that
\begin{align}\label{e:converse:eq:2}
|\nabla^\p_0 F_0|^2\leq E\left[|\nabla_0^\p F_\epsilon|^2+\kappa\int_0^\epsilon |\nabla_0^\p F_r| |\nabla_r^\p F_r|\, dr\right]\, ,
\end{align}
which by using \eqref{e:converse:1} again gives us
\begin{align}\label{e:converse:3}
|\nabla^\p_0 F_0|^2\leq E\left[|\nabla_0^\p F_\epsilon|^2\right]+\kappa\,\epsilon+o(\epsilon)\, .
\end{align}

Combining \eqref{e:converse:2} and \eqref{e:converse:3} we infer that
\begin{align}
\Ric(v,v)\geq -\kappa -\epsilon^{-1}o(\epsilon)\, ,	
\end{align}
which by limiting $\epsilon\to 0$ gives us our desired lower bound.\\

\noindent{\bf (C3) implies Upper Ricci.}  We have seen that cylinder functions of one variable capture the lower Ricci curvature bound, and therefore we necessarily need more complicated functions on path space to capture the upper Ricci curvature bound.  In fact, we will see that a cylinder function of two variables is enough.  For $x\in M$ and $v\in T_xM$ a unit vector let us choose a smooth compactly supported function $f_2:M\times M\to \dR$ such that
\begin{align}
f_2(x,x)=0\, , \;\;\;\; \nabla^{(1)} f_2(x,x) = 2v\, ,\;\;\;\; \nabla^{(2)} f_2(x,x) = -v\, ,\;\;\;\; \nabla^2 f_2(x,x)=0\, .	
\end{align}
For instance, we may choose $f_2(y,z)=2f_1(y)- f_1(z)$ where $f_1$ is defined in \eqref{e:converse:0}.  Let us then define the cylinder function $F_\epsilon(\gamma)\equiv f_2(\gamma(0),\gamma(\epsilon))$.  A computation tells us for $0<t\leq \epsilon$ that
\begin{align}\label{e:converse:6}
	&\nabla_0^\p F_t = \nabla^{(1)}f_2(x,\gamma(t))+P_{t}\nabla H^{(2)}_{\epsilon-t}f_2(x,\gamma(t))\, ,\;\;\;\; \nabla_t^\p F_t = P_{t}\nabla H^{(2)}_{\epsilon-t}f_2(x,\gamma(t))\, ,\notag\\
	&|\nabla_t^\p\nabla_0^\p F_t| \leq |\nabla^2 f_2|(x,\gamma(t))+|\nabla^2 H^{(2)}_{\epsilon-t}f_2|(x,\gamma(t))\, .
\end{align}

Note in particular that $\nabla_0^\p F_t\approx v$, $\nabla_t^\p F_t\approx -v$ and $|\nabla_t^\p\nabla_0^\p F_t|\approx 0$ for $\epsilon$ small, at least for a typical curve (again, one could be quite effective about this but it is unneccessary).  Note that in contrast to the test function in the lower Ricci context, we have that $\nabla_0^\p F_t$ and $\nabla_t^\p F_t$ have flipped signs.  Now using the generalized Bochner formula as in \eqref{e:converse:eq:1} we have that $t\mapsto |\nabla_0^\p F_t|^2-\int_0^t \big(|\nabla_r^\p\nabla_s^\p F_r|^2 + \Ric(\nabla_0^\p F_r,\nabla_r^\p F_r)\, dr\big)$ is a martingale and can compute
\begin{align}\label{e:converse:eq:3}
|\nabla^\p_0 F_0|^2=E\left[|\nabla_0^\p F_\epsilon|^2-\int_0^\epsilon \big(|\nabla_r^\p\nabla_s^\p F_r|^2 + \Ric(\nabla_0^\p F_r,\nabla_r^\p F_r)\big)\, dr\right]\, ,
\end{align}
which in combination with \eqref{e:converse:6} allows us to write
\begin{align}\label{e:converse:4}
|\nabla^\p_0 F_0|^2=E_0\left[|\nabla_0^\p F_\epsilon|^2\right]+\epsilon Rc(v,v)+o(\epsilon)\, .
\end{align}
We have used that $\nabla_0^\p F_t$ and $\nabla_t^\p F_t$ have opposite signs to obtain a positive sign in front of the Ricci term.  Additionally, using $(C3)$ as in \eqref{e:converse:eq:2} we arrive at \eqref{e:converse:3}:
\begin{align}\label{e:converse:5}
|\nabla^\p_0 F_0|^2\leq E_0\left[|\nabla_0^\p F_\epsilon|^2\right]+\kappa\,\epsilon+o(\epsilon)\, .
\end{align}

Combining \eqref{e:converse:4} and \eqref{e:converse:5} we infer that
\begin{align}
\Ric(v,v)\leq \kappa +\epsilon^{-1}o(\epsilon)\, ,	
\end{align}
which by limiting $\eps\to 0$ finishes the proof of the upper bound.

\vspace{1cm}

\noindent{\bf Other Converse Implications.}  Using the same test function, it is now straightforward to check that the other estimates on path space also imply the Ricci bound. To illustrate this in one more case, let us consider the gradient estimate (R3). Testing with a 1-point cylinder function we infer again that $\Ric\geq -\kappa g$. To prove the upper Ricci bound consider the test function
\begin{equation}
F_\eps(\gamma)= f_2 (\gamma(0),\gamma(\eps))
\end{equation}
as above. Expanding the gradient estimate (R3) gives
\begin{equation}\label{eq_final1}
\abs{\nabla_x E_x [F_\eps]}^2\leq \left(1+\frac{\kappa}{2}\eps\right) E\left[ \abs{\nabla_0^\p F_\eps}^2 + \frac{\kappa}{2}\eps \abs{\nabla_{\eps/2}^\p F_{\eps/2}}^2\right]+o(\eps)\leq 1+\kappa \eps + o(\eps)\, .
\end{equation}
On the other hand, as in \eqref{e:converse:4} from the generalized Bochner formula we see that
\begin{equation}\label{eq_final2}
\abs{\nabla_x E_x[F_\eps]}^2 \geq  1+ \eps\Ric(v,v) + o(\eps)\, .
\end{equation}
Combining \eqref{eq_final1} and \eqref{eq_final2} we conclude that $\Ric\leq \kappa g$. This finishes the proof of the converse implication, and thus the proof of Theorem \ref{thm_new_char}, Theorem \ref{thm_improved_grad_ricci} and Theorem \ref{thm_char_ricci}. $\square$\\

\begin{remark}
It is quite straightforward to plug in the test functions into all the estimates, but there are also several alternatives to close all the loops of implications, as we will briefly illustrate now. Applying the log-Sobolev inequality (R7) to $F^2=1+\eps G$ gives the Poincare inequality (R6).  Dividing the estimate (R6) by $\abs{t_1-t_0}$ and taking the limit $\abs{t_1-t_0}\to 0$ gives the quadratic variation estimate (R5). Moreover, using $d[F,F]_t=\abs{\nabla_t^\p F_t}^2 dt$ it is easy to see that $(R5)\Leftrightarrow (R3)$ and that $(R4)\Leftrightarrow (R2)$. And of course $(R2)\Rightarrow (R3)$ via H\"older exactly as in $(G2)\Rightarrow (G3)$. Summing up, if one doesn't want to plug a test function in any estimate other than (C3) and (R3), where we already did it, this is enough to close all the loops of equivalences.
\end{remark}

\bibliography{HN_ricci_martingales}

\bibliographystyle{alpha}

\vspace{10mm}
{\sc Robert Haslhofer, Department of Mathematics, University of Toronto, 40 St George Street, Toronto, ON M5S 2E4, Canada}\\

{\sc Aaron Naber, Department of Mathematics, Northwestern University, 2033 Sheridan Road, Evanston, IL 60208, USA}\\

\emph{E-mail:} roberth@math.toronto.edu, anaber@math.northwestern.edu

\end{document}